\documentclass[reqno,12pt,a4paper]{amsart}
\usepackage[a4paper,margin=1in,footskip=0.4in,heightrounded]{geometry}
\usepackage{lastpage}
\usepackage[shortlabels]{enumitem}
\usepackage{t1enc}
\usepackage[utf8]{inputenc}
\usepackage{graphicx}
\usepackage{tabularx,placeins}
\usepackage{caption,subcaption}
\usepackage{multirow}
\usepackage{tikz,color}
\usepackage{mathtools}
\usepackage{latexsym,amssymb}
\usepackage{afterpage}
\usetikzlibrary{arrows}
\graphicspath{ {./eps/} }
\tikzset{every loop/.style={}}
\usepackage[pagewise]{lineno} 
\usepackage{cleveref}

\newtheorem{prop}{Proposition}
\newtheorem{lem}{Lemma}

\newtheorem{con}{Conjecture}
\newtheorem{cor}{Corollary}
\newtheorem{theo}{Theorem}
\theoremstyle{definition}

\theoremstyle{remark}

\crefdefaultlabelformat{#2{\scshape #1}#3}
\crefname{subfigure}{fig.}{figs.}
\Crefname{subfigure}{Figure}{Figure}

\title[A self-consistent system with multiple acims.]{A self-consistent dynamical system with multiple absolutely continuous invariant measures.}
\author[F. M. S\'elley]{Fanni M. S\'elley}
\address{Alfr\'ed R\'enyi Institute of Mathematics \\
Reáltanoda u. 13-15, H-1053 Budapest, Hungary
}
\curraddr{Laboratoire de Probabilit\'es, Statistique et Mod\'elisation (LPSM), Sorbonne Universit\'e, Universit\'e de Paris,
	4 Place Jussieu, 75005 Paris, France}
\email{selley@lpsm.paris}
\date{\today}

\begin{document}

\begin{abstract}
In this paper we study a class of \emph{self-consistent dynamical systems}, self-consistent in the sense that the discrete time dynamics is different in each step depending on current statistics. The general framework admits popular examples such as coupled map systems. Motivated by an example of \cite{blank2017ergodic}, we concentrate on a special case where the dynamics in each step is a $\beta$-map with some $\beta \geq 2$. Included in the definition of $\beta$ is a parameter $\varepsilon > 0$ controlling the strength of self-consistency. We show such a self-consistent system which has a unique absolutely continuous invariant measure (acim) for $\varepsilon=0$, but at least two for any $\varepsilon > 0$. With a slight modification, we transform this system into one which produces a phase transition-like behavior: it has a unique acim for $0< \varepsilon < \varepsilon^*$, and multiple for sufficiently large values of $\varepsilon$. We discuss the stability of the invariant measures by the help of computer simulations employing the numerical representation of the self-consistent transfer operator. 
\end{abstract}

\maketitle

\let\thefootnote\relax
\let\thefootnote\relax

\section{Introduction}

A self-consistent dynamical system is a discrete time dynamical system where the dynamics is not the \emph{same map} in every time step, but computed by the \emph{same rule} from some momentary statistical property of the system. Such systems arise in problems of both physical and mathematical motivation, but their rigorous mathematical treatment so far has been restricted to some special cases, mainly \emph{coupled map systems}. 

Self-consistent systems bear resemblance to the larger framework of systems that are governed by laws that vary over time. The uniqueness and stability of the invariant measure is thoroughly studied for examples including \emph{non-autonomous dynamical systems} \cite{gora2019absolutely,ott2009memory}, \emph{random dynamical systems} \cite{arnold1998random,baladi1996random,bogenschutz2000stochastic,buzzi2000absolutely} and \emph{random perturbations of dynamical systems} \cite{blank1998random,baladi1993spectra}. However, a self-consistent system is not a special case of any of these examples, as the dynamics in each step is not chosen via an abstract rule or drawn randomly from a set of possibilities, but is computed in a deterministic way from the trajectory of an initial probability distribution on the phase space.    

The introduction of self-consistent systems dates back to \cite{kaneko1990globally}, who studied globally coupled interval maps. In a globally (or mean-field) coupled map system the dynamics is the composition of the individual dynamics of a single site and a coupling dynamics which is typically the identity perturbed by the \emph{mean-field} generated by the sites, hence self-consistency arises from coupling. The effect of the mean-field is usually multiplied by a nonnegative constant $\varepsilon$ called the coupling strength, which controls to what extent the self-consistency distorts the uncoupled dynamics. The literature studying coupled map systems is quite extensive. As systems of coupled maps are just loosely connected to the present work, we refrain from giving a complete bibliography, as a starting point see \cite{chazottes2005dynamics, phdthesis} and the references therein. Typically the existence and uniqueness of the invariant measure is studied in terms of the coupling strength. Most available results prove the uniqueness of the SRB measure for small coupling strength \cite{blank2011self,jiang1998equilibrium,keller2005spectral}, but in some specific models phase transition-like phenomena can also be observed \cite{bardet2009stochastically}: unique invariant measure for small coupling strength, and multiple for stronger coupling. 

The literature of self-consistent systems not arising from coupled map systems is particularly sparse (in fact the only example known to us is the one discussed below). In this paper our goal is to study such a system which is in some sense much simpler than a coupled map system, hence interesting phase transition-like phenomena can be shown by less involved methods than the ones used for example in \cite{bardet2009stochastically}. As results of this type are particularly hard to obtain in the coupled map setting, our results, although obtained in a simplified self-consistent system, contribute to the few existing examples.

Our main point of reference is Section 5 of \cite{blank2017ergodic}, specifically the two systems defined by Example 5.2 which we now recall. Let $X=[0,1]$ and $E_{\mu}=\int_{0}^1x\text{ d}\mu(x)$, where $\mu$ is a probability measure on $X$. Let
\begin{enumerate}[(a)]
	\item $T_{\mu}(x)=x \cdot E_{\mu}$,
	\item $T_{\mu}(x)=x \slash E_{\mu} \mod 1$ (where $1\slash0 \mod 1$ is defined as 0).
\end{enumerate}
The map $T_{\mu}$ induces an action on the space of probability measures, and an invariant measure of such a system is a probability measure for which $\mu=(T_{\mu})_*\mu$. As Blank noted, in case (a) the only invariant measures are the point masses supported on 0 and 1, as $T_{\mu}$ is a contracting linear map in all nontrivial case. Case (b) is more interesting since now $T_{\mu}$ is a particular piecewise expanding map, a \emph{beta map}, first studied by \cite{parry1960beta,parry1964representations,renyi1957representations}. Blank pointed out, that the self-consistent system has infinitely many mutually singular invariant measures, including the Lebesgue measure. We are going to show that this picture is not complete, as the existence of multiple Lebesgue-absolutely continuous invariant measures (acims) can be shown.

The stability of these invariant measures is a more delicate question. By \emph{stability} we mean that the invariant density attracts all elements of some neighborhood in a suitable norm, hence these equilibrium states rightfully describe an asymptotic behavior of the system. Rigorous results in this direction are only available in case of smooth self-consistent dynamics \cite{keller2000ergodic,balint2018synchronization} and the treatment of piecewise smooth dynamics (such as example (b) of Blank) would require a completely different approach. However, to obtain a rough picture of the phenomena to be expected, computer simulations can be very useful. Numerical approximations of transfer operators and invariant densities have been extensively studied in the last few decades, typically by the help of generalized Galerkin-type methods. The idea behind these discretization schemes is the construction of a sequence of finite rank operators approximating the transfer operator of the dynamical system. The most notable scheme is Ulam's method \cite{ulam1960collection}, a relatively crude but robust method. The convergence of the fixed points of the finite rank operators to the invariant density was first proved by \cite{li1976finite} in case of one-dimensional dynamics, and since then, many generalizations have followed. For a comprehensive study see \cite{murray1997discrete} and the references within. Better approximation can be achieved by higher order Galerkin-type methods \cite{ding1993high,ding1996finite}. For a more extensive survey of the discretization of the Perron-Frobenius operator see \cite{klus2016numerical}.  The approximation of invariant densities of \emph{non-autonomous} dynamical systems on the other hand have a much more limited literature, focusing mainly on the setting of random dynamical systems \cite{froyland1999ulam,froyland2014stability,froyland2019quenched}. The self-consistent case, to the best of our knowledge is an uncharted territory. To make the first steps, we consider specific systems (motivated by example (b) of Blank) with piecewise linear dynamics. The advantage of such systems is that the transfer operator maps the space of piecewise constant functions to itself, hence no discretization scheme is needed to compute pushforward densities. However, the task is not completely trivial, as the chaotic nature of the dynamics causes computational errors to blow up rapidly.    

The setting and our results are summarized in Section \ref{sec_results}. In Section \ref{sec_psi} we introduce an auxiliary function $\psi^{\varepsilon}$ providing the main tool for the proofs of our results. In Section \ref{sec_thm1} we study a self-consistent system which interpolates linearly between the doubling map and case (b) of Blank's example by a parameter $\varepsilon$. We show that the system has a unique acim only in the case of $\varepsilon=0$ (giving the doubling map) and has multiple absolutely continuous invariant measures for any $\varepsilon > 0$ (in particular for $\varepsilon=1$, giving Blank's example.) In Section \ref{sec_thm2} we study a modified version of this self-consistent system which indeed exhibits a phase transition like-behavior: it has a unique acim if $\varepsilon$ is smaller than some $\varepsilon^* > 0$ and multiple acims if $\varepsilon$ is sufficiently large. In Section \ref{sec_num} we showcase  some results of computer simulations intended to study the stability of invariant densities with respect to the iteration of the self-consistent transfer operator.

\textbf{Acknowledgments.} The research was supported
by the European Research Council (ERC) under the European Union's Horizon 2020 research and innovation programme (grant agreement No 787304) and by the Hungarian National Foundation for Scientific Research (NKFIH OTKA) grant K123782. The author would like to express gratitude to Péter Bálint, Imre Péter Tóth and Péter Koltai for helpful discussions. The author also expresses gratitude to an anonymous referee for providing a substantial amount of help to correct the proof of Lemma \ref{lem_lip}.

\section{The results} \label{sec_results}
Let $X=[0,1]$ and denote the space of probability measures on $X$ by $\mathcal{M}(X)$. For a measure $\mu \in \mathcal{M}(X)$, let 
\begin{equation} \label{eq_expected}
E_{\mu}=\int_{0}^1x\text{ d}\mu(x).
\end{equation}
Given an initial probability measure $\mu \in \mathcal{M}(X)$ and $\varepsilon \geq 0$, define the dynamics $T_{\mu}^{\varepsilon}: X \to [0,1)$ as 
\begin{equation}
T_{\mu}^{\varepsilon}(x)=\left(2+\varepsilon F\left(\frac{1}{E_{\mu}}-2\right)\right)x \mod 1, \quad x \in X,
\end{equation}
where $F \in C^1(\mathbb{R},\mathbb{R})$ is such $F(0)=0$. The parameter $\varepsilon$ controls to what extent the measure $\mu$ influences the dynamics. (If $\mu_0=\delta_0$ the Dirac mass concentrated on zero, we define $T_{\delta_0}^{\varepsilon} \equiv 0$.) In particular for $\varepsilon=0$, $\mu$ has no influence at all and
\begin{equation} \label{eq_doubling}
T_{\mu}^{0}(x)=2x \mod 1, \quad x \in X 
\end{equation}
for any $\mu \in \mathcal{M}(X)$.

We are going to study the \emph{self-consistent system} 
\begin{equation} \label{eq_selfc}
(\mathcal{M}(X),(T_{\cdot}^{\varepsilon})_*\cdot).
\end{equation}
An \emph{invariant measure} of the system \eqref{eq_selfc} is a measure $\mu \in \mathcal{M}(X)$ such that
\[
(T_{\mu}^{\varepsilon})_*\mu=\mu.
\] 
It is easy to see that the system \eqref{eq_selfc} has many invariant measures: for instance the Lebesgue measure $\lambda$, since
\[
E_{\lambda}=\int_{0}^1x\text{ d}x=\frac{1}{2} \Rightarrow T_{\lambda}^{\varepsilon}(x)=2x \mod 1
\]
implying that
\begin{equation} \label{eq_lebesgue}
(T_{\lambda}^{\varepsilon})_*\lambda=\lambda.
\end{equation}
In addition to this, infinitely many mutually singular invariant measures exist: consider a measure uniformly distributed on a periodic orbit of the doubling map which is symmetric about $\frac{1}{2}$. More precisely, consider $x \in [0,1]$ with binary expansion
\[
. \overline{\underbrace{111\dots1}
_{k \text{ times}}\underbrace{000\dots0}_{k \text{ times}}}
\] 
As the doubling map acts as a shift on binary expansions, the images of $x$ are
\begin{align*}
&. \overline{11\dots11
00\dots00} \qquad .\overline{00\dots00
11\dots11}  \\
&. \overline{11\dots10
00\dots01} \qquad .\overline{00\dots01
11\dots10} \\
&. \overline{11\dots00
00\dots11} \qquad .\overline{00\dots11
11\dots00}\\
& \vdots \\
&. \overline{10\dots00
01\dots11} \qquad .\overline{01\dots11
10\dots00} 
\end{align*}
We can see that this orbit is symmetric about $\frac{1}{2}$: if $y$ is in this orbit, then so is $1-y$. This implies that if $\mu=\frac{1}{k+1}\sum_{j=0}^k \delta_{2^{j}x \mod 1}$, we have $E_{\mu}=\frac{1}{2}$ and $T_{\mu}^{\varepsilon}(x)=2x \mod 1$, by which $(T_{\mu}^{\varepsilon})_*\mu=\mu$.

Our main question is if \eqref{eq_selfc} has multiple invariant measures absolutely continuous with respect to the Lebesgue measure (\emph{acim}s). 

For $\varepsilon=0$ we have seen that irrespective of the measure $\mu$, the dynamics is always the doubling map. So in this case we have a unique absolutely continuous invariant measure.

We first show that by taking the identity as $F$ (producing Blank's example for $\varepsilon=1$) this property is immediately lost as we introduce self-consistency.  

\begin{theo} \label{theo_main2}
	Consider the self-consistent system \eqref{eq_selfc} and suppose that $F(x)=x$. Then for any $\varepsilon > 0$, at least two acims exist: one is Lebesgue, and the other is equivalent to Lebesgue. 
\end{theo}

We then show that under some additional assumptions on $F$, the uniqueness of the acim persists for $\varepsilon$ small enough. But not indefinitely: we also show that for sufficiently strong self-consistency this is not the case, i.e. multiple acims exist if $\varepsilon$ is large enough. 

\begin{theo} \label{theo_main}
Consider the self-consistent system \eqref{eq_selfc} and suppose that $F(x) > 0$ for all $x \neq 0$ and $F'(x)=O\left(\frac{1}{|\log x|}\right)$ as $x \to 0+$.
\begin{enumerate}
	\item There exists an $\varepsilon^*_1 > 0$ such that for $0 \leq \varepsilon < \varepsilon^*_1$ the only acim is the Lebesgue measure. 
	\item There exists an $\varepsilon^*_ 2\geq \varepsilon^*_1$ such that for $\varepsilon \geq \varepsilon^*_2$ at least two acims exist: one is Lebesgue, and the other is equivalent to Lebesgue. 
\end{enumerate}
\end{theo}

An example of the function $F$ for which Theorem \ref{theo_main} holds is $F(x)=x^{2k}$, where $k \geq 1$ is an integer.

To discuss the stability of these invariant densities we performed a series of computer simulations. Based on the results of these computations we make two conjectures.

\begin{con} \label{con_th1}
In the setting of Theorem \eqref{theo_main2},
\begin{enumerate}
	\item In case of $\varepsilon=0$, the uniform density is stable with respect to the $BV$-norm $\|\cdot\|_{BV}=\text{var}(\cdot)+\int|\cdot|$.
	\item In case of $\varepsilon>0$, the uniform density is not stable with respect to the $L^1$-norm $\|\cdot\|_{L^1}=\int|\cdot|$ but there exists a different invariant density in $BV$ which is stable with respect to the $L^1$-norm.
\end{enumerate}
\end{con}

Part (1) of this conjecture is the stability of Lebesgue measure under the doubling map, which is well known, we included it in the conjecture to highlight the bifurcation phenomenon.  

\begin{con} \label{con_th2}
	In the setting of Theorem \eqref{theo_main}, the uniform density is stable with respect to the $BV$-norm for all $\varepsilon > 0$. Other acims are unstable with respect to the $BV$-norm. 
\end{con}

\section{The auxiliary function $\psi^{\varepsilon}$} \label{sec_psi}

The aim of this section is to give the definition and discuss some properties of an auxiliary function on which the proofs of Theorems \eqref{theo_main2} and \eqref{theo_main} rely. Let 
\begin{equation}
T_{\beta}(x)= \beta x \mod 1, \quad x \in [0,1]
\end{equation}
such that $\beta > 1$, and denote by $\mu_\beta$ the unique acim of the system. By the classical results of \cite{renyi1957representations} we in fact know that such a measure exists, and it is equivalent to the Lebesgue measure. Remember that by our notation \eqref{eq_expected}
\[
E_{\mu_{\beta}}=\int_0^1 x \text{ d}\mu_{\beta}(x).
\]
Define $\psi^{\varepsilon}: (1,\infty) \to \mathbb{R}$ as
\begin{equation} \label{eq_psi}
\psi^{\varepsilon}(\beta)=2+\varepsilon F\left(\frac{1}{E_{\mu_{\beta}}}-2\right).
\end{equation}
Suppose there exists a $\bar{\beta}$ such that $\psi^{\varepsilon}(\bar{\beta})=\bar{\beta}$. Notice that in this case $\mu_{\bar{\beta}}$ is an invariant measure of the self-consistent system \eqref{eq_selfc}. Indeed, 
\[
\bar{\beta}=2+\varepsilon F\left(\frac{1}{E_{\mu_{\bar{\beta}}}}-2\right) \Rightarrow T_{\mu_{\bar{\beta}}}(x)=\bar{\beta}x \mod 1,
\]
and this implies that
\[
(T_{\mu_{\bar{\beta}}})_*\mu_{\bar{\beta}}=\mu_{\bar{\beta}}.
\]
This shows that every fixed point of $\psi^{\varepsilon}$ gives rise to an absolutely continuous invariant measure of \eqref{eq_selfc}. Moreover, if $F \geq 0$, the absolutely continuous invariant measures of \eqref{eq_selfc} and the fixed points of $\psi^{\varepsilon}$  are in one--to--one correspondence. So it suffices to study number of fixed points of the function $\psi^{\varepsilon}$ to prove our Theorems \ref{theo_main2} and \ref{theo_main}. 

We now state a lemma implying regularity properties of $\psi^{\varepsilon}$.

\begin{lem} \label{lem_psicont}
	1. Let $2 \leq \beta < 3$. There exists a constant $K > 0$ such that
	\[
	|E_{\mu_{\beta}}-E_{\mu_{2}}| \leq K |\beta-2|(1+|\log|\beta-2||).
	\]
	2. Let $2 < a_0 \leq a_1$. There exists a constant $K(a_0,a_1) > 0$ such that for all $\beta,\beta' \in [a_0,a_1]$,
	\[
	|E_{\mu_{\beta}}-E_{\mu_{\beta'}}| \leq K(a_0,a_1) |\beta-\beta'|(1+|\log|\beta-\beta'||).
	\]
\end{lem}

\begin{proof}
As proved in \cite{parry1964representations}, the \emph{unnormalized} invariant density of $T_{\beta}$ can be given by the formula
\begin{equation} \label{eq_hbeta}
h_{\beta}(x)=\sum_{n=0}^{\infty}\frac{1}{\beta^n}\mathbf{1}_{[0,T^n_{\beta}(1))}(x).
\end{equation}
Thus
\begin{align}
\int_0^1h_{\beta}(x)\text{ d}x&=\sum_{n=0}^{\infty}\int_0^1\frac{1}{\beta^n}\mathbf{1}_{[0,T^n_{\beta}(1))}(x)\text{ d}x=\sum_{n=0}^{\infty}\frac{1}{\beta^n}T_{\beta}^n(1) \label{eq_normalizer} \\
\int_0^1xh_{\beta}(x)\text{ d}x&=\sum_{n=0}^{\infty}\int_0^1\frac{1}{\beta^n}x\mathbf{1}_{[0,T^n_{\beta}(1))}(x)\text{ d}x=\sum_{n=0}^{\infty}\frac{1}{2\beta^n}(T_{\beta}^n(1))^2. \label{eq_exp}
\end{align} 
Notice in particular that since $T_2$ has range $[0,1)$, we have $T^0_{2}(1)=1$ and $T^n_2(1)=0$ for all $n \geq 1$, hence $h_2(x)=1$. This implies that 
\[
\int_0^1h_{2}(x)\text{ d}x=1 \quad \text{and} \quad \int_0^1xh_{2}(x)\text{ d}x=\frac{1}{2}.
\]

By definition and the above observation, we have 
\[
E_{\mu_{\beta}}=\frac{1}{\int_0^1h_{\beta}(x)\text{ d}x}\int_0^1xh_{\beta}(x)\text{ d}x \quad \text{and} \quad E_{\mu_{2}}=\frac{1}{2}.
\]
We can deduce the following uniform bounds on $E_{\mu_{\beta}}$:
\begin{align} \label{eq_embound}
\frac{1}{4} \leq \frac{1}{2}\cdot \frac{\beta-1}{\beta}\leq E_{\mu_{\beta}} \leq \frac{1}{2},
\end{align}
where the last inequality is always strict if $\beta$ is not an integer. Indeed, 
\begin{align*}
\frac{1}{2}\frac{\sum_{n=0}^{\infty}\frac{1}{\beta^n}(T_{\beta}^n(1))^2}{\sum_{n=0}^{\infty}\frac{1}{\beta^n}T_{\beta}^n(1)} \leq \frac{1}{2},
\end{align*}
since $(T_{\beta}^n(1))^2 \leq T_{\beta}^n(1)$ for all $n$. In particular, if $\beta$ is an integer, $T_{\beta}^0(1)=1$ and $T_{\beta}^n(1)=0$ for all $n \geq 1$, hence the equality (and this is the only case that equality can occur). We get the lower bound from the fact $1 \leq h_{\beta} \leq \frac{\beta}{\beta-1}$ (including just the first term in the sum \eqref{eq_hbeta} versus all terms). 

In particular observe that
\begin{equation} \label{eq_expvalueineq}
E_{\mu_{\beta}} \leq E_{\mu_{2}}
\end{equation}
so $|E_{\mu_{\beta}}-E_{\mu_{2}}|=E_{\mu_{2}}-E_{\mu_{\beta}}$.

Let us write $h_{\beta}(x)=1+g_{\beta}(x)$ where 
\[
g_{\beta}(x)=\sum_{n=1}^{\infty}\frac{1}{\beta^n}\mathbf{1}_{[0,T^n_{\beta}(1))}(x).
\] 
Then
\[
E_{\mu_{\beta}}=\frac{\int_0^1 x(1+g_\beta(x)) \mathrm{d}x}{\int_0^1 1+g_\beta(x) \mathrm{d}x}=\frac{\frac12+\int_0^1 x g_\beta(x) \mathrm{d}x}{1+\int_0^1 g_\beta(x) \mathrm{d}x},
\]
so
\begin{align*}
E_{\mu_\beta}&\ge \frac{\frac12}{1+\int_0^1 g_\beta(x) \mathrm{d}x}\ge \frac12 \left( 1-\int_0^1 g_\beta(x) \mathrm{d}x\right) \Rightarrow E_{\mu_{2}}-E_{\mu_{\beta}} \leq \frac{1}{2}\int_0^1 g_\beta(x) \mathrm{d}x, 
\end{align*}
so it is enough to show
\begin{equation} \label{eq_loglip}
\int_0^1 g_\beta(x) \mathrm{d}x=\sum_{n=1}^{\infty}\frac{1}{\beta^n}T^n_{\beta}(1) \leq N\eta+\left(\frac{1}{\beta}\right)^N
\end{equation}
for $\eta=\beta-2$ and $N=\lceil \frac{\log \eta}{\log 1/\beta} \rceil$ to prove the claim of part (a). To show \eqref{eq_loglip}, we first note that
\begin{equation} \label{eq_loglip2}
\sum_{k=N+1}^{\infty}\frac{1}{\beta^k}= \frac{1}{\beta-1}\left(\frac{1}{\beta}\right)^N
\end{equation}
Now notice that if $T^{n-1}_{\beta}(1)=\beta^{n-2}(\beta-2) <\frac{1}{\beta}$, the images $T^{n}_{\beta}(1)$ fall under the first branch. By the choice of $N$, this is exactly the case for all $k=1,\dots,N$. So
\[
\frac{1}{\beta^n}T^n_{\beta}(1)=\frac{\beta^{n-1}(\beta-2)}{\beta^n} \leq \frac{1}{\beta}(\beta-2) \qquad \text{for }1 \leq n \leq N
\] 
implying
\begin{equation} \label{eq_loglip3}
\sum_{k=1}^{N}\frac{1}{\beta^k} \leq \frac{N}{\beta}(\beta-2)
\end{equation}
Putting together \eqref{eq_loglip2} and \eqref{eq_loglip3} we get \eqref{eq_loglip}.

Part 2 can be proved in an analogous way to \cite[Proposition 2]{keller2008continuity} as a consequence of \cite[Corollary 1]{keller2008continuity}.
\end{proof} 

This lemma has the following important corollary:

\begin{cor} \label{cor_psi}
	1. For $2 \leq \beta < 3$, there exists a constant $\tilde{K} > 0$ such that
	\[
	\frac{1}{E_{\mu_{\beta}}}-2 \leq \tilde{K} |\beta-2|(1+|\log|\beta-2||).
	\]
	2. $\psi^{\varepsilon}(\beta)$ is continuous for all $\beta \geq 2$.
\end{cor}

Indeed, by part (1) of Lemma \ref{lem_psicont} and \eqref{eq_embound}
\begin{align*}
\frac{1}{E_{\mu_{\beta}}}-\frac{1}{E_{\mu_2}}= \left| \frac{E_{\mu_{2}}-E_{\mu_{\beta}}}{E_{\mu_{\beta}}E_{\mu_{2}}}\right| \leq 8K |\beta-2|(1+|\log|\beta-2||).
\end{align*}
We note that part (2) of Lemma \ref{lem_psicont} serves only the purpose to draw the conclusion of part (2) of the above corollary.

We now outline our main idea behind the proofs of Theorems \ref{theo_main2} and \ref{theo_main}. First observe that $\psi^{\varepsilon}(2)=2$, since Lebesgue measure is an invariant measure of the doubling map. But since Lebesgue is invariant for any $\beta$-map where $\beta$ is an integer, we have $\psi^{\varepsilon}(k)=2$ for all $k \geq 2$, $k \in \mathbb{N}$. So if for some $\bar{\beta} \in (k,k+1)$ we have $\psi^{\varepsilon}(\bar{\beta}) > \bar{\beta}$, we can conclude that $\psi^{\varepsilon}$ has a fixed point $\beta^* \in (k,k+1)$. This implies that $\mu_{\beta^*}$ is an invariant measure of the self-consistent system \eqref{eq_selfc} that is equivalent, but not equal to Lebesgue. 

In the following sections we are going to prove Theorem \ref{theo_main} part (1) by showing that no such $\bar{\beta}$ exists for the stated values of $\varepsilon$, while we prove Theorem \ref{theo_main2} and Theorem \ref{theo_main} part (2) by showing the existence of a $\bar{\beta}$ such that $\psi^{\varepsilon}(\bar{\beta}) > \bar{\beta}$ in the settings considered.   

\section{Proof of Theorem \ref{theo_main2}} \label{sec_thm1}

In this section we consider the special case when $F$ is the identity. Now \eqref{eq_psi} takes the form
\[
\psi^{\varepsilon}(\beta)=2+\varepsilon \left( \frac{1}{E_{\mu_{\beta}}}-2\right).
\]
We are going to show that for any $\varepsilon > 0$ the map $\psi^{\varepsilon}$ has another fixed point in addition to $\beta=2$. This implies Theorem \ref{theo_main2}, as discussed in Section \ref{sec_psi}. According to our arguments in Section \ref{sec_psi}, it is more than enough to prove the following proposition:

\begin{prop} \label{prop_id}
Let $\varepsilon > 0$. For any $\delta > 0$ there exists a $\beta > 2$, $|\beta-2|<\delta$ such that
\begin{equation} \label{eq_idwant}
\psi^{\varepsilon}(\beta) > \beta. 
\end{equation}
\end{prop}

This is the consequence of the following lemma, which claims that the log-Lipschitz continuity of $\beta \mapsto E_{\mu_{\beta}}$ at $\beta=2$ stated in Lemma~\ref{lem_psicont} cannot be improved to Lipschitz continuity:

\begin{lem} \label{lem_id}
There exists a sequence $\beta_k \to 2+$ such that
\[
|E_{\mu_{\beta_k}}-E_{\mu_{2}}| > c_{\beta_k}|\beta_k-2|
\]
for some $c_{\beta_k} > 0$ such that $\lim_{\beta_k \to 2+}c_{\beta_k}=\infty$.
\end{lem}

To see that Lemma~\ref{lem_id} readily implies \eqref{eq_idwant}, note first that we can discard the absolute values from the statement of this lemma by \eqref{eq_expvalueineq}. As we previously showed that $E_{\mu_{\beta}}$ is bounded away from 0, we also have
\[
\frac{1}{E_{\mu_{\beta_k}}}-\frac{1}{E_{\mu_{2}}} > \tilde{c}_{\beta_k}(\beta_k-2)
\]
for some $\tilde{c}_{\beta_k} > 0$ such that $\lim_{\beta_k \to 2+}\tilde{c}_{\beta_k}=\infty$.
But then for $\tilde{c}_{\beta_k} > \frac{1}{\varepsilon}$ we have
\[
\psi^{\varepsilon}(\beta_k)=2+\varepsilon\left( \frac{1}{E_{\mu_{\beta_k}}}-2\right) > 2+\varepsilon\tilde{c}_{\beta_k}(\beta_k-2) > \beta_k.
\]
By choosing $k_0$ so large such that $\tilde{c}_{\beta_{k_0}} > \frac{1}{\varepsilon}$ holds for all $k \geq k_0$, we obtain that $\psi^{\varepsilon}(\beta_k) > k$ for all $k \geq k_0$. 

\begin{proof}[The proof of Lemma \ref{lem_id}.]
We are going to construct the sequence $\beta_k$ explicitly. Let $\beta_k$ be such that the first $k$ images of 1 fall under the first branch of $T_{\beta_k}$ and the $k+1$-th image of 1 is 0. This means that
\begin{align}
T_{\beta_k}(1)&=\beta_k-2 \nonumber \\
T_{\beta_k}^2(1)&=\beta_k(\beta_k-2) \nonumber \\
T_{\beta_k}^3(1)&=\beta_k^2(\beta_k-2) \nonumber \\
&\vdots \nonumber\\
T_{\beta_k}^k(1)&=\beta_k^{k-1}(\beta_k-2) \nonumber \\
T_{\beta_k}^{k+1}(1)&=\beta_k^{k}(\beta_k-2)=1 \equiv 0. \label{eq_betak}
\end{align}
To obtain $\beta_k$, one simply has to find the unique positive solution of \eqref{eq_betak}. It is easy to see that the thus defined $\beta_k \to 2+$ as $k \to \infty$. 

\begin{figure}
	\centering
	\begin{tikzpicture}[scale=2]
	\draw (0,0) -- (2,0) node[above right] {$x$};
	\draw (0,0) -- (0,2) node[above right] {\hspace{0.1cm}$T_{\beta_3}(x)$};
	\draw[dashed] (2,0) -- (2,2) -- (0,2);
	\draw[dotted] (0.95,0) -- (0.95,2);
	\draw[dotted] (1.9,0) -- (1.9,2);
	\draw[thick] (0,0) -- (0.95,2);
	\draw[thick] (0.95,0) -- (1.9,2);
	\draw[thick] (1.9,0) -- (2,0.21);
	\draw (0.04,-0.2) node {\text{\tiny{$\beta_3-2$}}};
	\draw (0.55,0.2) node {\text{\tiny{$\beta_3(\beta_3-2$)}}};
	\draw (1.25,-0.2) node {\text{\tiny{$\beta_3^2(\beta_3-2)$}}};
	\draw (1.08,-0.4) node {\text{\tiny{$=\frac{1}{\beta_3}$}}};
	\draw (2,-0.4) node {\text{\tiny{$=1$}}};
	\foreach \x/\xtext in {2/\text{\tiny{\hspace{1cm}$\beta_3^3(\beta_3-2)$}},0.95/,0.21/,0.44/}
	\draw[shift={(\x,0)}] (0pt,2pt) -- (0pt,-2pt) node[below] {$\xtext$};
	\foreach \y/\ytext in {2/1}
	\draw[shift={(0,\y)}] (2pt,0pt) -- (-2pt,0pt) node[left] {$\ytext$};
	\end{tikzpicture}
	\caption{The choice of $\beta_k$, $k=3$ is pictured.} \label{fig_betak}
\end{figure}
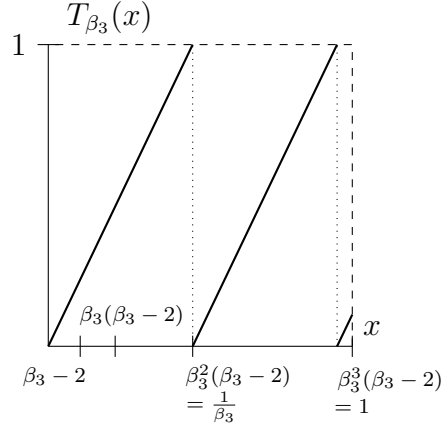

Straightforward calculations using \eqref{eq_normalizer} and \eqref{eq_exp} give that
\begin{align*}
\int_0^1h_{\beta_k}(x)\text{ d}x&=\sum_{n=0}^{\infty}\frac{1}{\beta_k^n}T_{\beta_k}^n(1)=1+k\cdot\frac{1}{\beta_k}(\beta_k-2) \\
\int_0^1xh_{\beta_k}(x)\text{ d}x&=\sum_{n=0}^{\infty}\frac{1}{2\beta_k^n}(T_{\beta_k}^n(1))^2=\frac{1}{2}\left(1+\frac{\beta_k^k-1}{\beta_k-1}\cdot\frac{1}{\beta_k}(\beta_k-2)^2\right)
\end{align*}
Now as
\[
E_{\beta_k}=\frac{1}{\int_0^1h_{\beta_k}(x)\text{ d}x}\int_0^1xh_{\beta_k}(x)\text{ d}x,
\]
and in particular $E_{2}=\frac{1}{2}$, we obtain that
\begin{align*}
|E_{\beta_k}-E_2|=\left|\frac{\frac{1}{2}\cdot\left(1+\frac{\beta_k^k-1}{\beta_k-1}\cdot\frac{1}{\beta_k}(\beta_k-2)^2\right)}{1+k\cdot\frac{1}{\beta_k}(\beta_k-2)}-\frac{1}{2}\right|=|\beta_k-2|\cdot \left|\frac{\frac{\beta_k^k-1}{\beta_k-1}(\beta_k-2)-k}{2\beta_k+2k(\beta_k-2)}\right|.
\end{align*}
Since $\left|\frac{\frac{\beta_k^k-1}{\beta_k-1}(\beta_k-2)-k}{2\beta_k+2k(\beta_k-2)}\right| \to \infty$ we can choose
\[
c_{\beta_k}=\left|\frac{\frac{\beta_k^k-1}{\beta_k-1}(\beta_k-2)-k}{2\beta_k+2k(\beta_k-2)}\right|.
\] 
\end{proof}

\section{Proof of Theorem \ref{theo_main}} \label{sec_thm2}
Throughout this section we assume that $F > 0$ for all $x \neq 0$ and $F'$ has the property $F'(x)=O\left(\frac{1}{|\log x|}\right)$ as $x \to 0+$.

\subsection{Weak self-consistency: unique acim}
In this section we prove the following proposition:
\begin{prop}
There exists an $\varepsilon^*_1 > 0$ such that for $0 \leq \varepsilon < \varepsilon^*_1$ the function $\psi^{\varepsilon}$ has a unique fixed point.
\end{prop}

As discussed, this implies part (1) of Theorem~\ref{theo_main}.

\begin{proof}
As we previously noted $\psi^{\varepsilon}(2)=2$. We are going to show that no other fixed point exists provided that $\varepsilon$ is small enough.

First notice that $\psi^{\varepsilon} \geq 2$ by \eqref{eq_expvalueineq}, so no fixed point exists which is smaller that 2.

We now show that despite the irregularity of $\beta \mapsto E_{\mu_{\beta}}$ in $\beta=2$ (as stated in Lemma \ref{lem_id}), $\psi^{\varepsilon}$ is Lipschitz continuous in $\beta=2$ as a result of the derivative of $F$ vanishing in 0 at an appropriate rate. 

\begin{lem} \label{lem_lip}
There exists a $\delta=\delta(F) > 0$ and an $M=M(F) > 0$  such that
\begin{equation} \label{eq_psilip}
|\psi^{\varepsilon}(\beta)-\psi^{\varepsilon}(2)| \leq \varepsilon M |\beta-2|
\end{equation}
for all $\beta \in (2,2+\delta)$.	
\end{lem}

\begin{proof}
Let $G: (0,h) \to (0,\infty)$ (for some small $h>0$) be the inverse function of the map $x \mapsto \tilde{K}x(1+|\log x|)$ defined for small $x > 0$, where $\tilde{K}$ is from Corollary \ref{cor_psi} part (1), that is, 
\[
\frac{1}{E_{\mu_{\beta}}}-2 \leq \tilde{K}(\beta-2)(1+|\log(\beta-2)|)
\] 
for all $\beta > 2$ sufficiently close to 2. Taking arbitrary $\alpha > 1$, since $x^{\frac{1}{\alpha}} > \tilde{K}x(1+|\log x|)$ for all sufficiently small $x > 0$, we have that $x^{\alpha} < G(x)$ and so $\frac{1}{\alpha|\log x|} < \frac{1}{|\log G(x)|}=\tilde{K}G'(x)$ for all sufficiently small $x > 0$. Therefore, since by assumption $F'(x)$ is $O\left(\frac{1}{|\log x|} \right)$ as $x \to 0+$ we have that $F'(x) < MG'(x)$ for some constant $M=M(F)$ and for sufficiently small $x>0$. Since $F(0)=0$, it follows that $F(x) < M G(x)$ for sufficiently small $x>0$ and so
\[
\psi^{\varepsilon}(\beta)-2=\varepsilon F\left(\frac{1}{E_{\mu_{\beta}}}-2 \right) \leq M \varepsilon G\left(\frac{1}{E_{\mu_{\beta}}}-2 \right) \leq \varepsilon M(\beta-2).
\]

\end{proof}
Now we show that for small enough $\varepsilon$, we have $\psi^{\varepsilon}(\beta) < \beta$ for $2 < \beta < 2+\delta$. By \eqref{eq_psilip}, 
\begin{align*}
\psi^{\varepsilon}(\beta)-2 &\leq \varepsilon M   (\beta-2) \\
\psi^{\varepsilon}(\beta) &\leq \varepsilon M   (\beta-2)+2
\end{align*}
and
\begin{align*}
\varepsilon  M  (\beta-2)+2 & < \beta \quad \text{whenever} \\
\varepsilon   &< \frac{1}{M}.
\end{align*}
We now show
\begin{align} \label{eq_bigbeta}
\psi^{\varepsilon}(\beta)=2+\varepsilon F\left(\frac{1}{E_{\mu_{\beta}}}-2\right) < \beta \quad \text{for $\beta \in [2+\delta,\infty)$.}
\end{align}
By \eqref{eq_embound} we have
\[
F\left(\frac{1}{E_{\mu_{\beta}}}-2\right) \leq F\left(\frac{2}{\beta-1}\right).
\]
Since for $\beta \geq 2+\delta$
\[
2+\varepsilon F\left(\frac{1}{E_{\mu_{\beta}}}-2\right) \leq 2+\varepsilon F\left(\frac{2}{1+\delta}\right) < 2+\delta \leq \beta
\]
when $\varepsilon < \frac{\delta}{F\left(\frac{2}{1+\delta}\right)}$, \eqref{eq_bigbeta} is proved. This implies the statement of the proposition with
\[
\varepsilon^*_1=\min \left\{\frac{1}{M},\frac{\delta}{F\left(\frac{2}{1+\delta}\right)}\right\}.
\]
\end{proof}

\subsection{Strong self-consistency: multiple acims}

In this section we prove Theorem \eqref{theo_main}, part (2). For this it suffices to find a single $\bar{\beta}$ such that $\psi^{\varepsilon}(\bar{\beta}) > \bar{\beta}$ holds for sufficiently large $\varepsilon$.

We in fact show that we can achieve $\psi^{\varepsilon}(\bar{\beta}) > \bar{\beta}$ for arbitrary $\bar{\beta} \in (k,k+1)$, provided that $\varepsilon$ is large enough in terms of $\bar{\beta}$.   

\begin{prop} \label{prop_strong}
Let $\bar{\beta} \in (k,k+1)$. There exists an $\varepsilon^*_2=\varepsilon^*_2(\bar{\beta})$ such that 
\[
\psi^{\varepsilon}(\bar{\beta}) > \bar{\beta}
\]
holds for all $\varepsilon > \varepsilon^*_2(\bar{\beta})$.
\end{prop} 

\begin{proof}
We would like to have
\[
\psi^{\varepsilon}(\bar{\beta})=2+\varepsilon F\left(\frac{1}{E_{\mu_{\bar{\beta}}}}-2\right) > \bar{\beta}
\]
This holds if
\[
\varepsilon > \frac{\bar{\beta}-2}{F\left(\frac{1}{E_{\mu_{\bar{\beta}}}}-2\right)},
\]
so by the choice of 
\[
\varepsilon_2^*=\frac{\bar{\beta}-2}{F\left(\frac{1}{E_{\mu_{\bar{\beta}}}}-2\right)}
\]
the proposition is proved.
\end{proof}

This proposition has the following corollary:

\begin{cor} \label{cor_many}
	The self-consistent system \eqref{eq_selfc} can have an arbitrarily large finite number of invariant measures equivalent to the Lebesgue measure, provided that $\varepsilon$ is large enough. 
\end{cor}

To see this, let $\ell \geq 2$ be an arbitrary integer. Choose $\bar{\beta}_k \in (k,k+1)$, $k=2,\dots,\ell$. Let $\varepsilon^*=\max_{k \in \{2,\dots,\ell\}}\varepsilon_2^*(\bar{\beta}_k)$. Then Proposition~\ref{prop_strong} implies that for $\varepsilon > \varepsilon^*$, the function $\psi^{\varepsilon}$ has a fixed point on each of the intervals $(k,k+1)$, $k=2,\dots,\ell$ implying a total number of at least $\ell$ acims.

\section{Numerical results} \label{sec_num}

\subsection{Illustrations of $\beta \mapsto \psi^{\varepsilon}(\beta)$} \label{sec_num1}
To illustrate the results of the previous sections, we present some computer approximations of the curve $\beta \mapsto \psi^{\varepsilon}(\beta)$ for some appropriate functions $F$. Since $T_{\beta}(x)=\beta x \mod 1$ is ergodic (proved in \cite{renyi1957representations}), we can approximate $E_{\mu_{\beta}}$ by computing the reciprocal of ergodic averages. This means we can approximate the graph of $\psi^{\varepsilon}$ by
\[
\psi^{\varepsilon}(\beta) \approx P_{x,N}(\beta)=2+\varepsilon F\left(\frac{N}{\sum_{n=0}^{N-1}T_{\beta}^n(x)}-2\right)
\] 
for (Lebesgue) almost every $x \in [0,1]$ and $N$ large. 

\begin{figure}
	\centering
	\begin{subfigure}[b]{0.44\textwidth}
		\centering
		\includegraphics[scale=0.5]{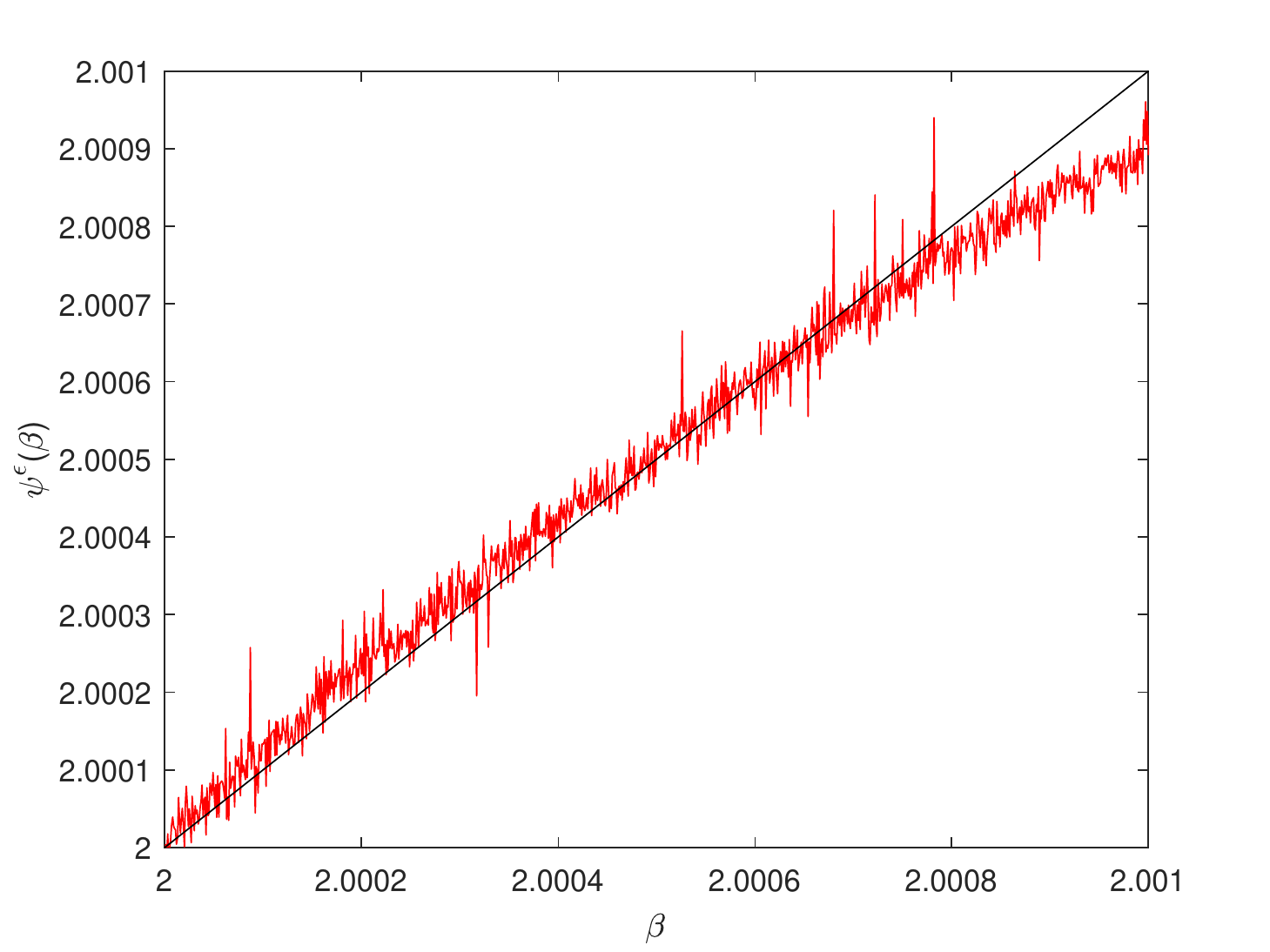}
		\caption{$N=10^8$, $\Delta=10^{-6}$, $\beta_1=2.001$,\\ $\varepsilon=0.1$}
	\end{subfigure}
	\qquad
	\begin{subfigure}[b]{0.44\textwidth}
		\centering
		\includegraphics[scale=0.5]{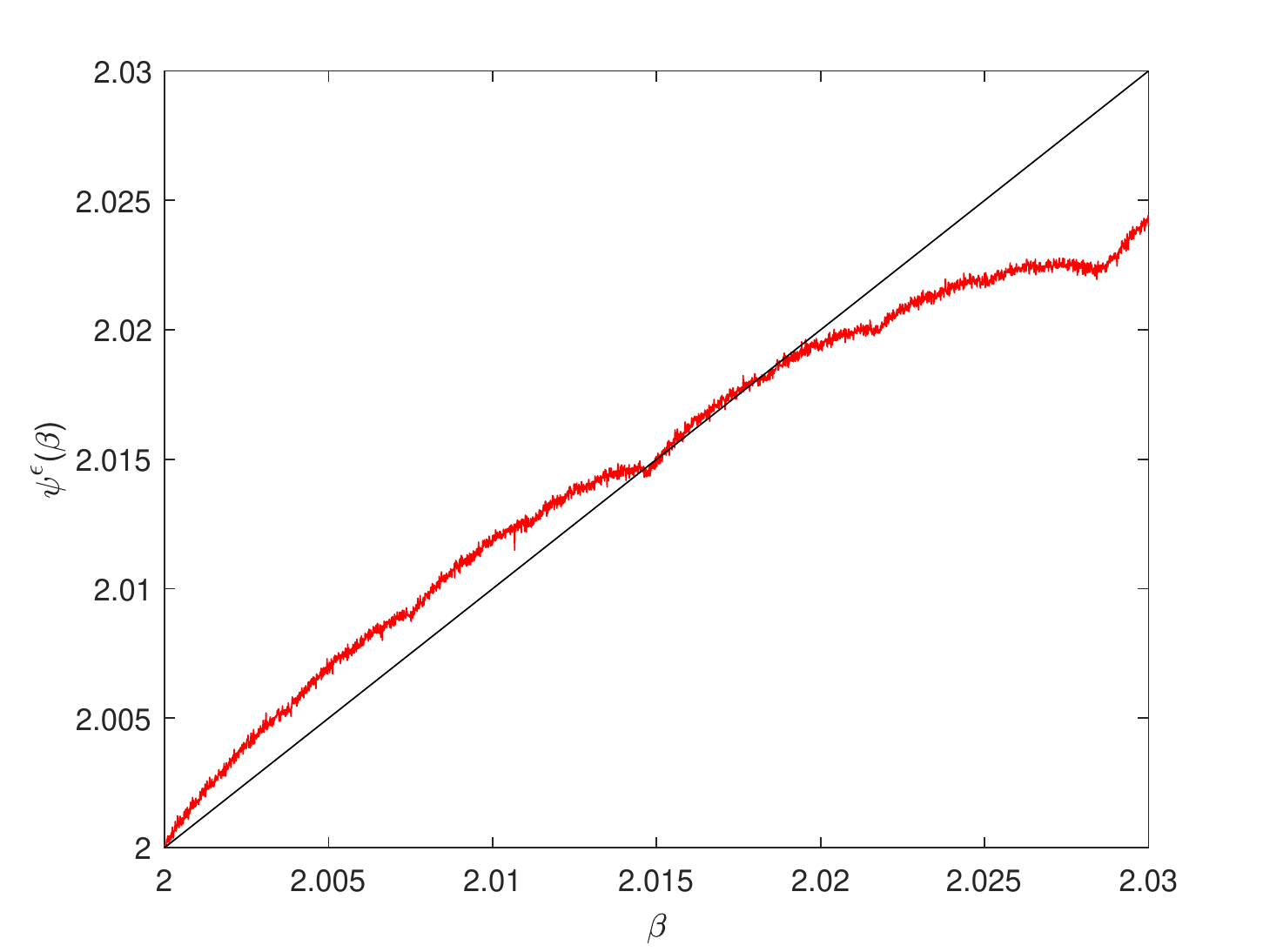}
		\caption{$N=10^7$, $\Delta=10^{-5}$, $\beta_1=2.03$,\\ $\varepsilon=0.2$}
	\end{subfigure}
	\caption{Approximation of $\psi^{\varepsilon}$ with ergodic averages, $F(x)=x$. $P_{x,N}(\beta)$ is plotted for $\beta \in [2,\beta_1]$ with gridsize $\Delta$ and for $x$ drawn uniform randomly from $[0,1]$. The line $x=y$ is plotted in black.} \label{fig_psi0}
\end{figure}

We first consider the setting of Theorem \ref{theo_main2}: the case when $F$ is the identity. We illustrate on Figure \ref{fig_psi0} that no matter how small $\varepsilon$ is, the curve approximating $\beta \mapsto \psi^{\varepsilon}(\beta)$ always grows above the line $x=y$ for $\beta$ values sufficiently close to 2. This shows that $\beta \mapsto \frac{1}{E_{\mu_{\beta}}}-2$ cannot be Lipschitz at 2, otherwise multiplying it with sufficiently small $\varepsilon$ would force the curve $\beta \mapsto \psi^{\varepsilon}(\beta)=2+\varepsilon\left( \frac{1}{E_{\mu_{\beta}}}-2\right)$ to always stay below $x=y$.

\begin{figure}[h!!]
	\centering
	\begin{subfigure}[b]{\textwidth}
		\centering
		\includegraphics[scale=0.5]{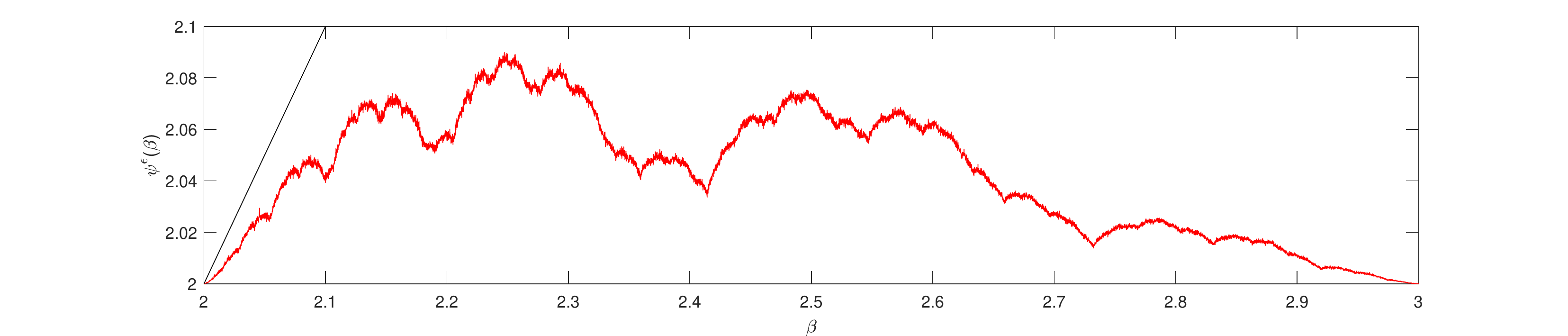}
		\caption{$F(x)=x^2$} \label{fig_psi1}
	\end{subfigure}
	\qquad
	\begin{subfigure}[b]{\textwidth}
		\centering
		\includegraphics[scale=0.5]{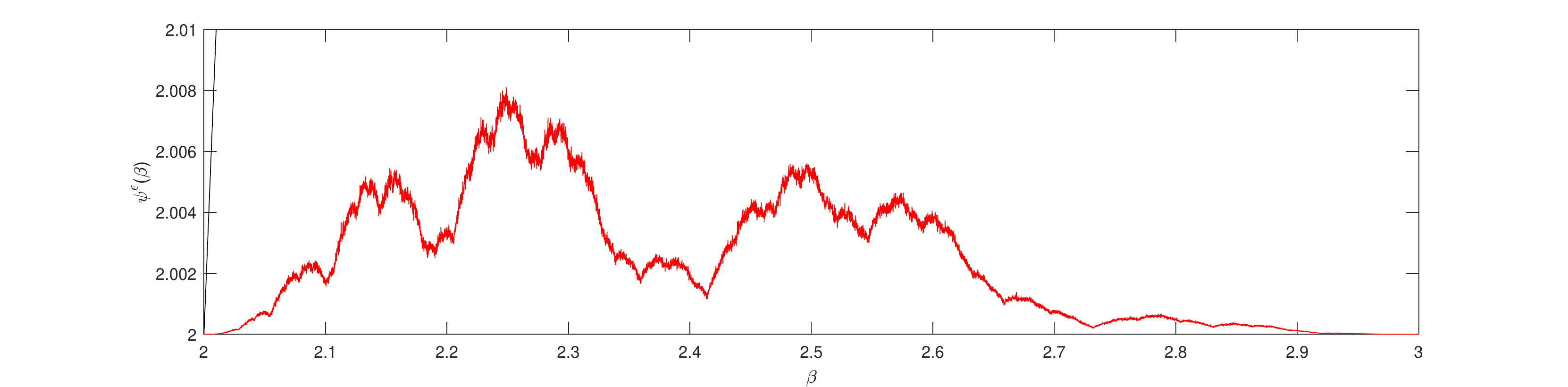}
		\caption{$F(x)=x^4$} \label{fig_psi3}
		\qquad
		\begin{subfigure}[b]{\textwidth}
			\centering
			\includegraphics[scale=0.5]{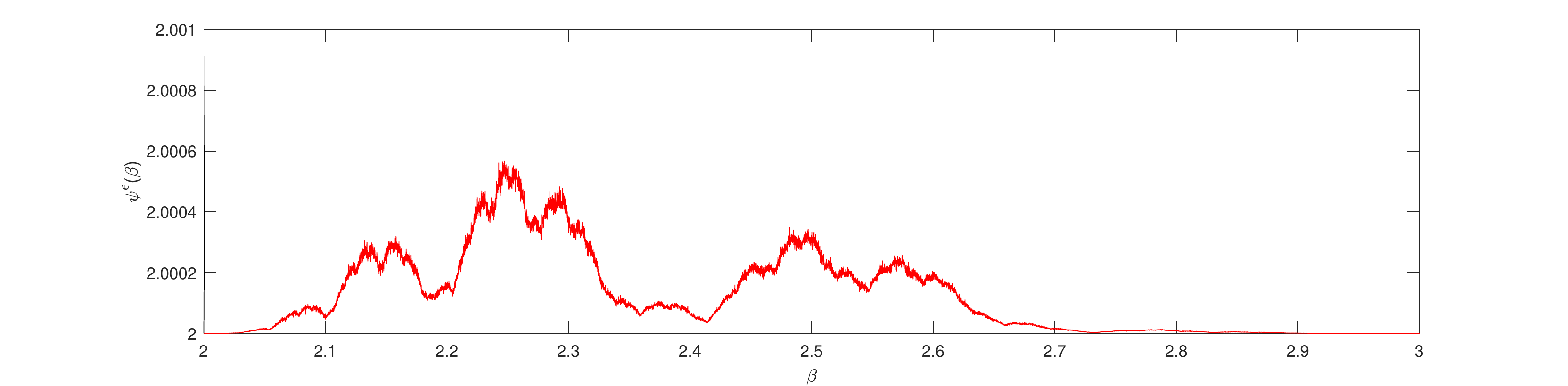}
			\caption{$F(x)=x^6$} \label{fig_psi2}
		\end{subfigure}
	\end{subfigure}
	\caption{Approximation of $\psi^{\varepsilon}$ with ergodic averages. $P_{x,N}(\beta)$ is plotted for $\beta \in [2,3]$ with gridsize $\Delta=10^{-4}$, $N=10^6$ and for $x$ drawn uniform randomly from $[0,1]$,  $\varepsilon=0.8$. The line $x=y$ is plotted in black.} \label{fig_psi}
\end{figure}

We now consider the setting of Theorem \ref{theo_main}. We first study the setting of part (1), that is when $\varepsilon$ is sufficiently small. Now $\beta \mapsto F\left(\frac{1}{E_{\mu_{\beta}}}-2\right)$ is Lipschitz at 2 as a result of $F'(x)=O\left(\frac{1}{|\log x|}\right)$ and the log-Lipschitz continuity of $\beta \mapsto \frac{1}{E_{\mu_{\beta}}}-2$. So sufficiently small $\varepsilon$ will not let the curve $\beta \mapsto \psi^{\varepsilon}(\beta)$ rise above $x=y$.  The results of our computer simulations are pictured on Figure \ref{fig_psi}, showing clearly that the curve approximating $\beta \mapsto \psi^{\varepsilon}(\beta)$ has a single intersection with the diagonal $x=y$ at $\beta=2$.

\begin{figure}[h!!]
	\centering
	\begin{subfigure}[b]{0.44\textwidth}
		\centering
		\includegraphics[scale=0.5]{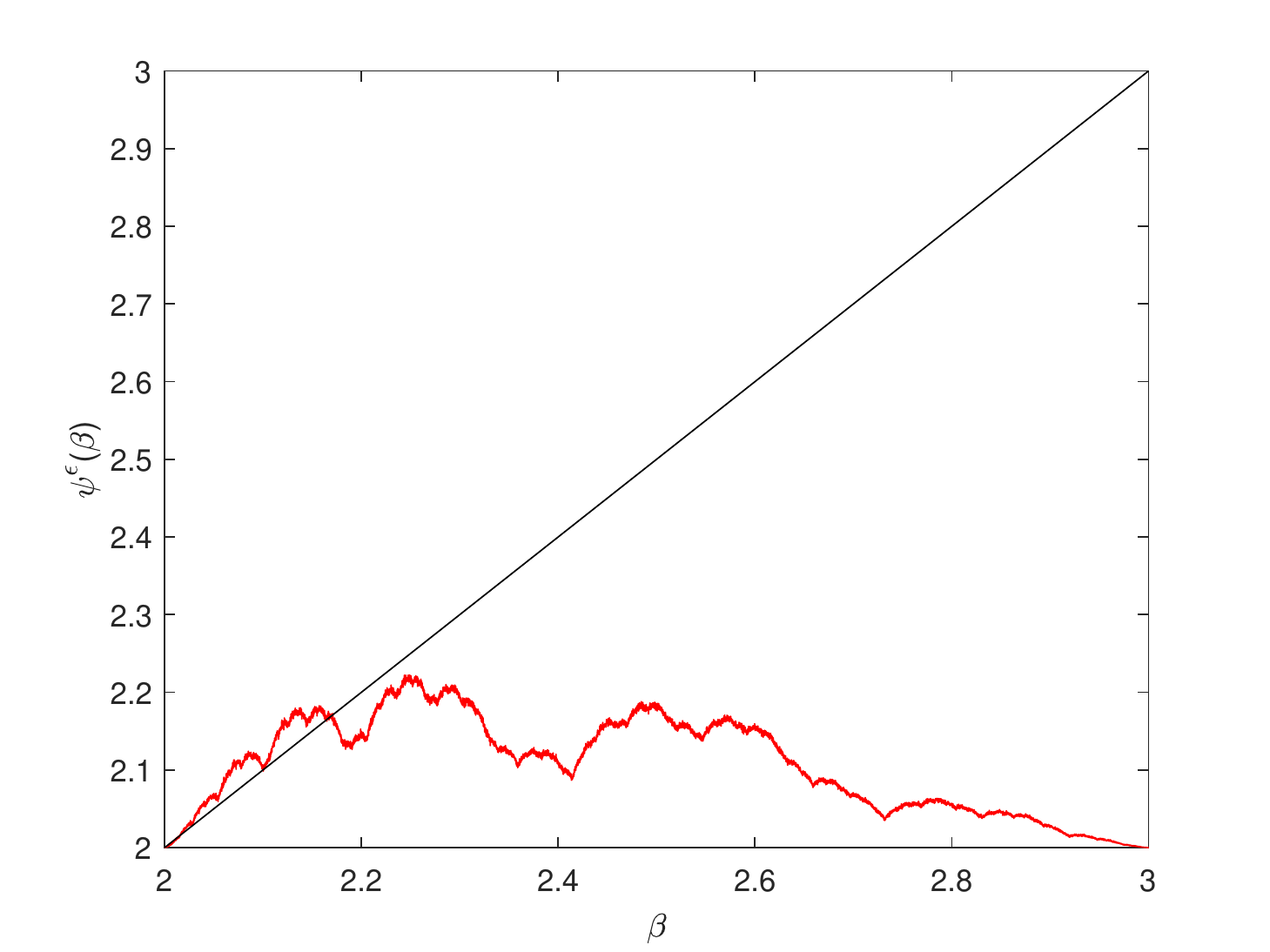}
		\caption{$\varepsilon=2.5$, $\beta_1=3$} 
	\end{subfigure}
	\qquad
	\begin{subfigure}[b]{0.44\textwidth}
		\centering
		\includegraphics[scale=0.5]{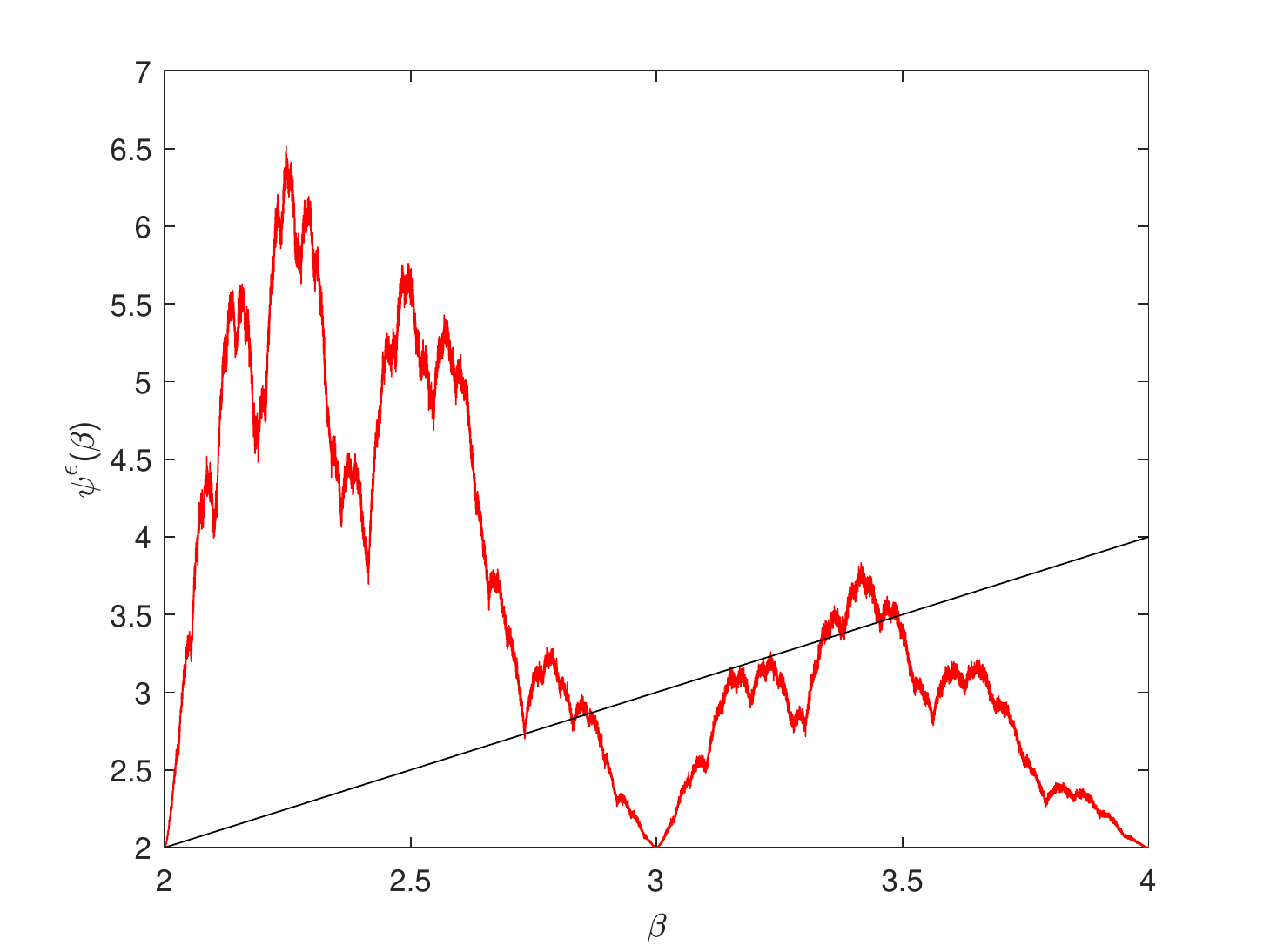}
		\caption{$\varepsilon=50$, $\beta_1=4$} 
	\end{subfigure}
	\qquad
	\begin{subfigure}[b]{0.44\textwidth}
		\centering
		\includegraphics[scale=0.5]{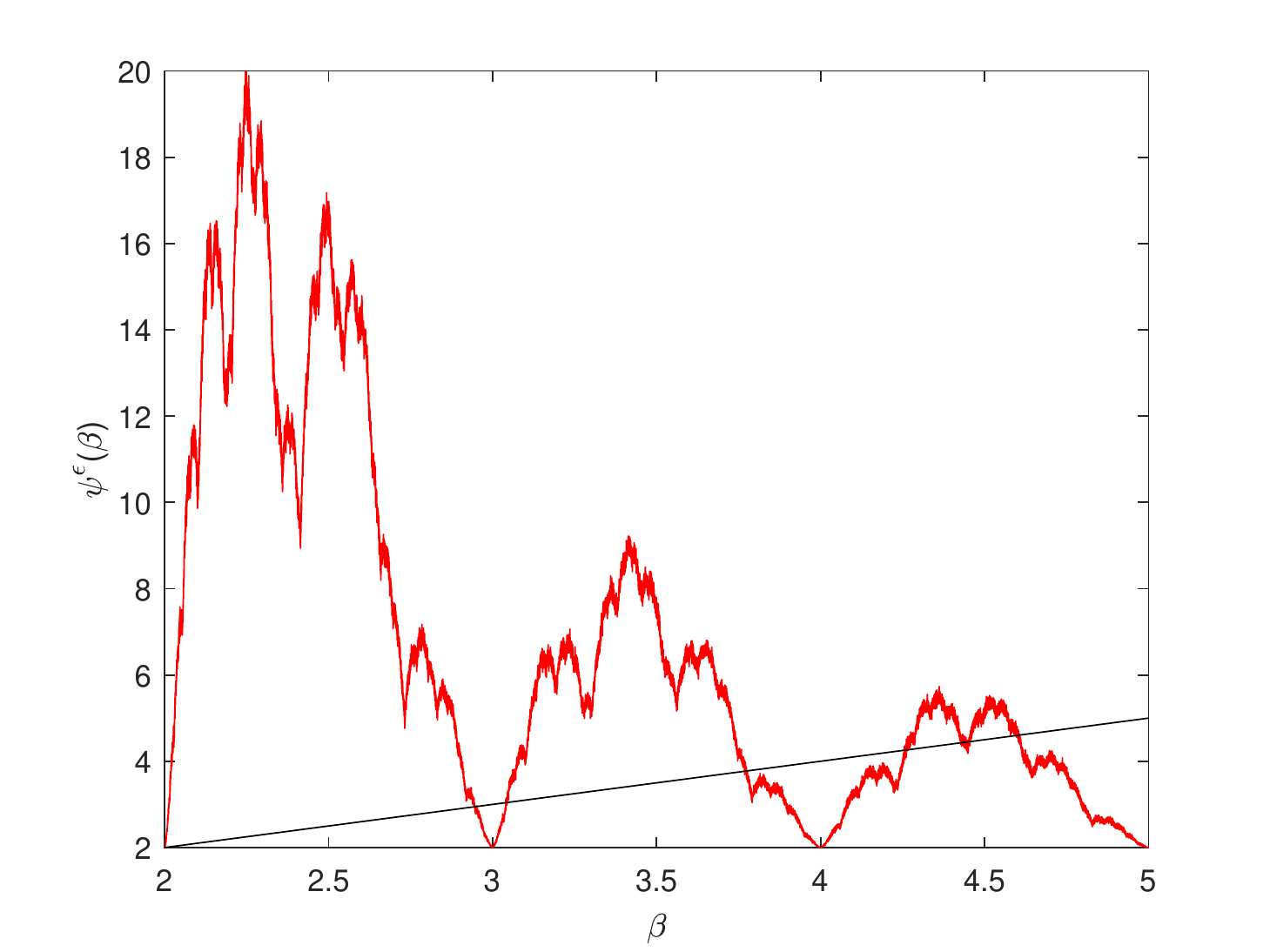}
		\caption{$\varepsilon=200$, $\beta_1=5$} 
	\end{subfigure}
	\caption{Approximation of $\psi^{\varepsilon}$ with ergodic averages, $F(x)=x^2$. $P_{x,N}(\beta)$ is plotted for $\beta \in [2,\beta_1]$ with gridsize $\Delta=10^{-4}$, $N=10^6$ and for $x$ drawn uniform randomly from $[0,1]$. The line $x=y$ is plotted in black.} \label{fig_psi20}
\end{figure}

The setting of Theorem \ref{theo_main} part (2) is studied on Figure \ref{fig_psi20} first for the special case $F(x)=x^2$. We can clearly see, as suggested by Corollary \ref{cor_many}, that if $\varepsilon$ is larger and larger, the curve approximating $\beta \mapsto \psi^{\varepsilon}(\beta)$ has intersections with the line $x=y$ on more and more intervals between two consecutive integers.

\begin{figure}[h!!]
	\centering
	\begin{subfigure}[b]{0.44\textwidth}
		\centering
		\includegraphics[scale=0.5]{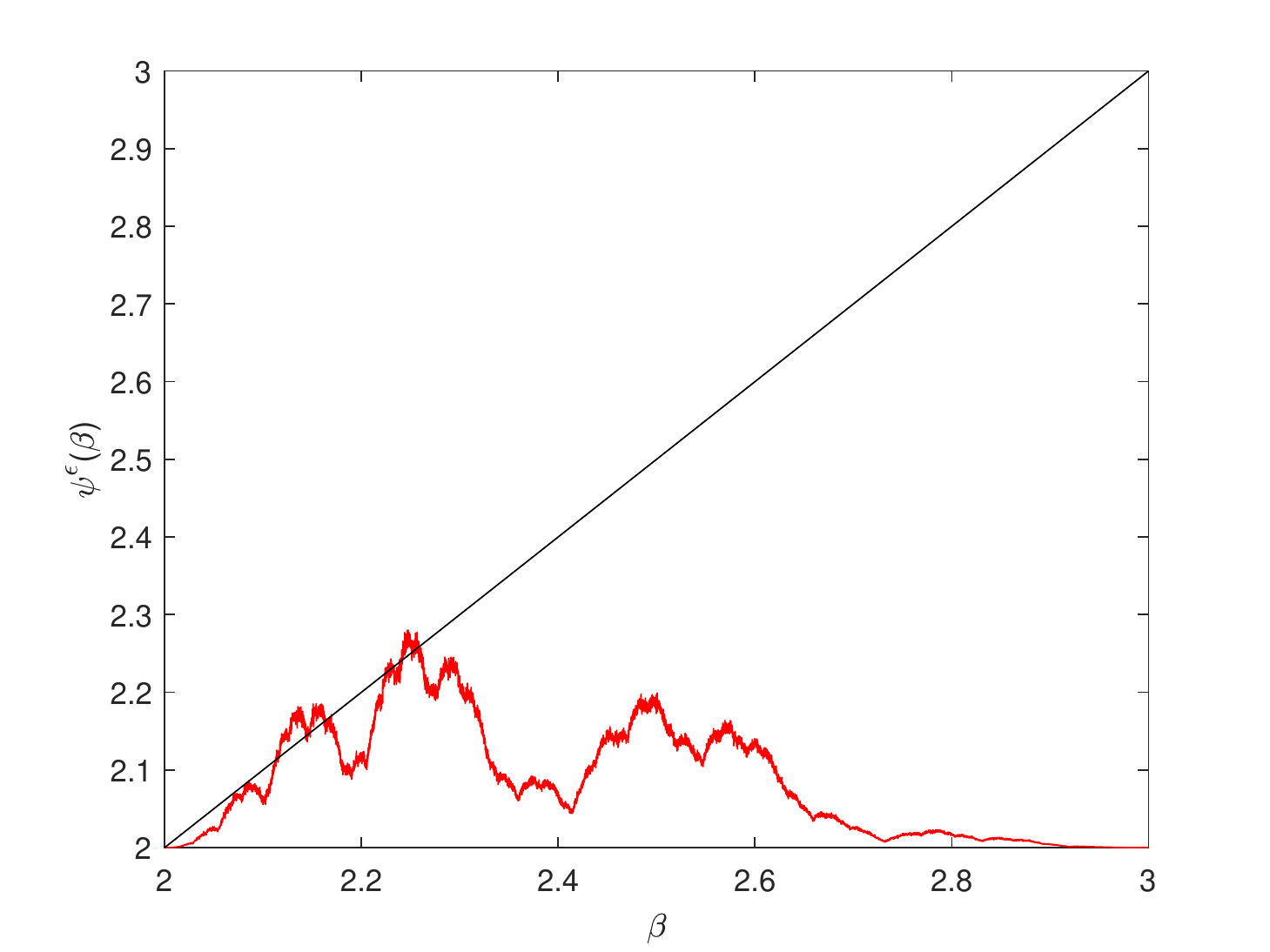}
		\caption{$\varepsilon=35$, $F(x)=x^4$} 
	\end{subfigure}
	\qquad
	\begin{subfigure}[b]{0.44\textwidth}
		\centering
		\includegraphics[scale=0.5]{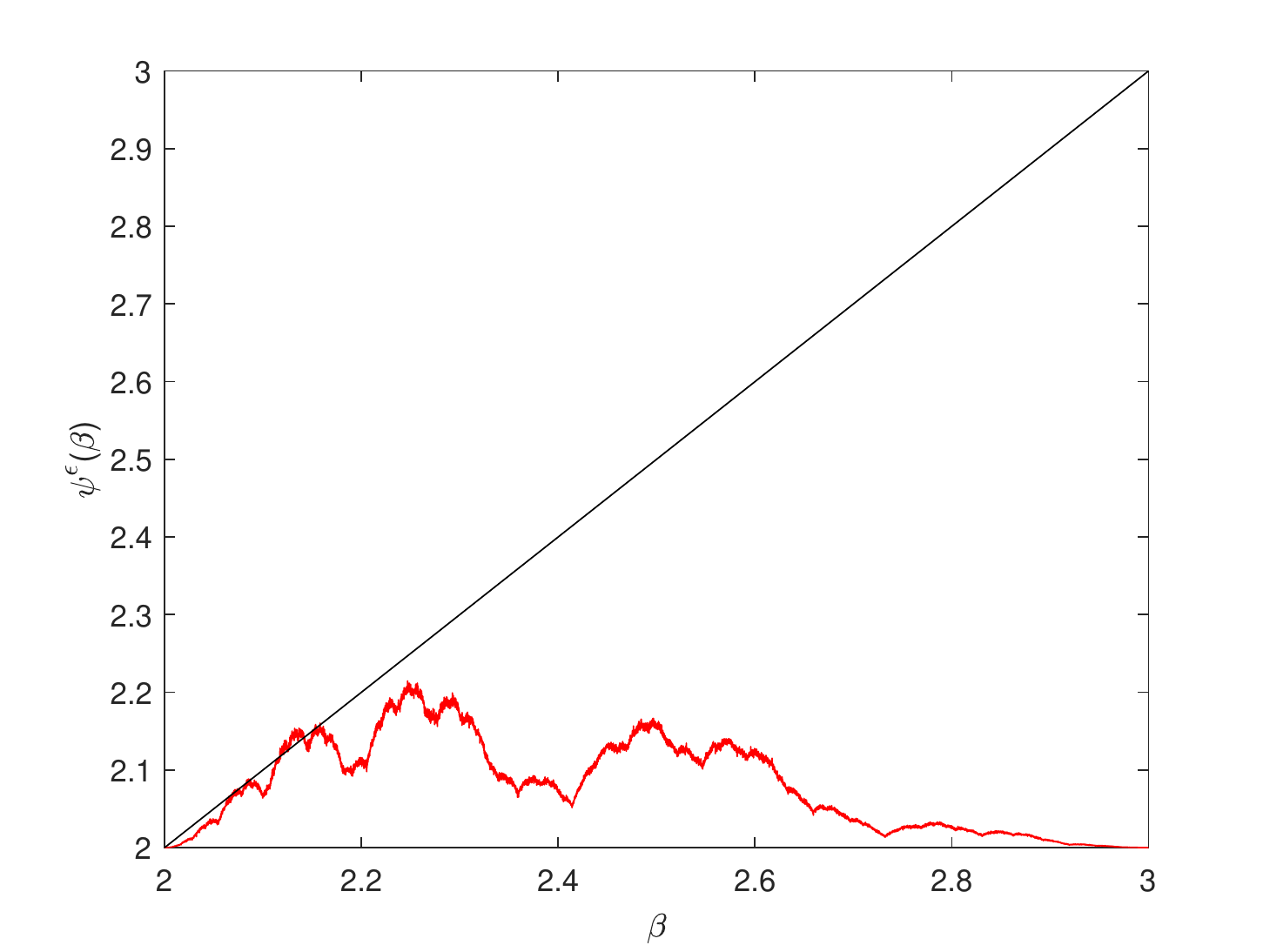}
		\caption{$\varepsilon=400$, $F(x)=x^6$} 
	\end{subfigure}
	\caption{Approximation of $\psi^{\varepsilon}$ with ergodic averages, $P_{x,N}(\beta)$ is plotted for $\beta \in [2,3]$ with gridsize $\Delta=10^{-4}$, $N=10^6$ and for $x$ drawn uniform randomly from $[0,1]$. The line $x=y$ is plotted in black.} \label{fig_psi30}
\end{figure}

Similar plots can be made for $F(x)=x^4$ and $F(x)=x^6$. On Figure \ref{fig_psi30} we can see that for sufficiently large $\varepsilon$ the curve approximating $\beta \mapsto \psi^{\varepsilon}(\beta)$ has intersections with the line $x=y$, indicating multiple invariant densities.

\subsection{Stability of the invariant densities} \label{sec_num2}

Although the existence of a unique or multiple invariant measures is an interesting phenomenon on its own, it is natural to further inquire about their stability. Let d$\mu=f$d$\lambda$ and let $P_{f}^{\varepsilon}$ denote the \emph{transfer operator} associated to the dynamics $T_{\mu}^{\varepsilon}$. This operator maps the density $h$ of a measure $\nu$ to the density of the measure $(T_{\mu}^{\varepsilon})_*\nu$.  Consider an initial measure d$\mu_0=f_0$d$\lambda$. The associated dynamics is $T_{\mu_0}^{\varepsilon}$ and the associated transfer operator is $P_{f_0}^{\varepsilon}$. Let $f_1=P_{f_0}^{\varepsilon}f_0$ be the pushforward density and d$\mu_1=f_1$d$\lambda$ be the pushforward measure. Continuing this further we obtain the $n$th step pushforward density as
\[
f_{n+1}=P_{f_n}^{\varepsilon}f_n, \qquad n=0,1,\dots
\]  
As an ease of notation, consider the \emph{self-consistent transfer operator} $\mathcal{F}_{\varepsilon}$ defined as 
\[
\mathcal{F}_{\varepsilon}(f_n)=P_{f_n}^{\varepsilon}f_n.
\]
An invariant density of the self-consistent system is a fixed point of this operator, and we can study its stability, that is, if densities sufficiently close to it in some metric converge to it. A function space well suited to this problem is the space $BV$ of functions of bounded variation, and the metric is the one given by the bounded variation norm.  

In case of $\varepsilon=0$, producing the doubling map as the dynamics in each step, it is well known that the Lebesgue measure is stable in the sense that all measures with a density of bounded variation converge to it exponentially fast. It is a natural question to ask if this still holds in our self-consistent system when Lebesgue is the unique absolutely continuous invariant measure, so in the setting of Theorem \ref{theo_main} part (1). On the other hand, in the setting of Theorem \ref{theo_main2} and Theorem \ref{theo_main} part (2) we have proved that the Lebesgue measure is a unique invariant measure of the system for sufficiently small values of $\varepsilon$, but for larger values multiple invariant measures exits. It would be interesting to see what kind of bifurcation occurs at the critical value of the coupling: is Lebesgue stable for small $\varepsilon$, and does it stay that way when multiple invariant measures arise, or does it lose its stability? Also, are the new invariant measures stable or unstable? For example, it would be interesting to show a similar behavior to the pitchfork bifurcation observed in the system of coupled fractional linear maps of \cite{bardet2009stochastically}: they show that the stable, unique invariant measure loses stability if the coupling strength is sufficiently increased, and two new stable invariant measures arise. 

To study these questions we present the results of some computer simulations. We first note that there exists an explicit expression for the transfer operator associated to the $\beta$ map $T_{\beta}(x)=\beta x \mod 1$:
\[
P_{\beta}f(x)=\sum_{y \in T^{-1}_{\beta}(x)}\frac{f(y)}{|T_{\beta}'(y)|}, \qquad f \in L^1,
\]
or more explicitly one can write
\begin{equation} \label{eq_pfbeta}
P_{\beta}f(x)=\begin{cases}
\frac{1}{\beta}\sum_{k=0}^{\lfloor \beta \rfloor}f\left( \frac{x+k}{\beta} \right) & \text{ if } \quad  0 \leq x < \beta-\lfloor \beta \rfloor \\
\frac{1}{\beta}\sum_{k=0}^{\lfloor \beta \rfloor-1}f\left( \frac{x+k}{\beta} \right) & \text{ if } \quad \beta-\lfloor \beta \rfloor \leq x \leq 1.  \\
\end{cases}
\end{equation}
It is clear that if $f$ is a a finite linear combination of indicators of intervals (a step function), $P_{\beta}f$ is also a step function. As functions of this kind are easy to store and manipulate by computer programs, we will restrict to working with densities of this kind. Note that this restriction means that we do not need to apply a discretization scheme to compute pushforward densities.

Let the step function $f$ be represented by the vectors $x_f \in [0,1]^N$ such that $x_f(1)=0$ and $x_f(N)=1$ and $y_f=(\mathbb{R}^+_0)^{N-1}$. The vector $x_f$ contains the jumps of the step function in increasing order and $y_f$ contains the respective heights of the steps. We define the \emph{total variation} of $f$ as 
\[
\text{var}(f)=\sum_{i=1}^{N-2}|y_f(i+1)-y_f(i)|.
\] 
Our initial densities will be generated in the following way: $x_f(2),\dots,x_f(N-1)$ are $N-2$ numbers drawn from the uniform random distribution on $(0,1)$ and then ordered increasingly. The values $\tilde{y}_f(1),\dots,\tilde{y}_f(N-1)$ are also drawn from the uniform random distribution on $(0,1)$. We define
\[
y_f(i)=\frac{1}{\sum_{j=1}^{N-1}\tilde{y}_f(j)(x_f(j+1)-x_f(j))}\cdot\tilde{y}_f(i), \quad i=1,\dots, N-1,
\] 
as this is an easy way to generate a fairly general step function of integral 1. 

Let the `expectation' associated to the density $f$ be defined as
\[
E_f=\sum_{i=1}^{N-1}\frac{y_f(i)}{2}(x_f(i+1)^2-x_f(i)^2).
\] 
The action of the self-consistent transfer operator $\mathcal{F}_{\varepsilon}(f)$ can be computed in the following way: take an initial density represented by the pair of vectors $(x_{f_0},y_{f_0}) \in [0,1]^{N_0} \times (\mathbb{R}^+_0)^{N_0-1}$. Compute $\beta_0=2+\varepsilon F \left(\frac{1}{E_{f_0}}-2 \right)$ defining the dynamics $T_{\beta_0}$. Compute the new vector of jumps $x_{f_1}$ as the vector containing the values $T_{\beta_0}(x_{f_0}(i))$, $i=1,\dots,N_0$ and 1 in increasing order.

To compute the new heights of the steps, choose $\rho > 0$ such that $\rho < \min|x_{f_1}(i+1)-x_{f_1}(i)|$ and compute
\[
y_{f_1}(i)=P_{\beta_0}f_0(x_{f_1}(i+1)-\rho) \quad \text{for} \quad i=1,\dots,N_0  
\]
by using the formula \eqref{eq_pfbeta}. As $f_0$ is piecewise constant and exactly stored, the evaluation of a $f_0$ at prescribed places is not an issue. Set $N_1=N_0+1$ and repeat with $(x_{f_1},y_{f_1}) \in [0,1]^{N_1} \times (\mathbb{R}^+_0)^{N_1-1}$. 

Our procedure is the following: we generate $K_1 \times K_2$ of the above described step functions in the following way: for each $k=1,\dots,K_1$ we generate a random integer between 1 and $M$ (denote this by $m(k)$) and generate $K_2$ step functions with $m(k)$ inner jumps (this means not counting 0 and 1). This gives us a fairly general pool of initial densities. 

We then compute the $T$ long trajectories of the densities with respect to the self consistent transfer operator $\mathcal{F}_{\varepsilon}$. We are going to use the notation $f^i_t=\mathcal{F}_{\varepsilon}f^i_0$ where the lower index refers to time and the upper index to which one of the initial densities we are considering, so $t=1,\dots,T$ and $i=1,\dots, K_1 \times K_2$. 

However, our experience is that computational errors grow rapidly and seriously skew our results in the long run. So in each iteration we normalize $f_t^i$ by the numerical integral
\[
I(f_t^i)=\sum_{j=1}^N y_{f_t^i}(j)(x_{f_t^i}(j+1)-x_{f_t^i}(j)),
\]
defining
\[
\tilde{f}_t^i=\frac{1}{I(f_t^i)}f_t^i.
\]
This assures us that the Perron-Frobenius operator is indeed applied to a density function. This density $\tilde{f}_t^i$ might not be the actual density $\mathcal{F}_{\varepsilon}^t(f_0^i)$, but one close to it. It can be thought of as $\mathcal{F}_{\varepsilon}^t(g_0^i)$ for some different initial density $g_0^i$ by anticipating a type of shadowing property.   

We are going to study two mean quantities of the densities. Define the \emph{mean slope} of the densities at time $t$ as
\[
\overline{\beta}(t)=\frac{1}{K_1\times K_2}\sum_{i=1}^{K_1\times K_2}2+\varepsilon F\left(\frac{1}{E_{\tilde{f}_t^i}}-2\right)
\]
and the \emph{mean total variation}  as 
\[
\overline{\text{var}}(t)=\frac{1}{K_1\times K_2}\sum_{i=1}^{K_1\times K_2}\text{var}(\tilde{f}^i_t).
\] 
Finally, we define our notions of \emph{stability} of an invariant density $f_*(\varepsilon)$. We are going to study two types of stability.
\begin{enumerate}
	\item \textbf{Stability in the BV-norm:} $$\overline{\text{var}}_{f_*(\varepsilon)}(t)=\frac{1}{K_1\times K_2}\sum_{i=1}^{K_1\times K_2}\text{var}(\tilde{f}^i_t-f_*(\varepsilon)) \to 0 \quad  \text{as} \quad  t \to \infty$$
	\item \textbf{Stability in the $L^1$-norm:} $$\overline{\text{int}}_{f_*(\varepsilon)}(t)=\frac{1}{K_1\times K_2}\sum_{i=1}^{K_1\times K_2}\int_{[0,1]}|\tilde{f}^i_t-f_*(\varepsilon)|d\lambda \to 0 \quad  \text{as} \quad  t \to \infty$$
\end{enumerate}

If an explicit expression for $f_*(\varepsilon)$ is not available, then in practice $f_*(\varepsilon)$ is $\tilde{f}_{\bar{T}}(\varepsilon)$ for some $\bar{T}$ considerably larger than $T$ and some fixed initial density, as we assume this is a good approximation of an invariant density. 

Note that $L^1$-stability is weaker than $BV$-stability. We further remark that $\overline{\text{var}}(t) \to 0$ as $t \to \infty$ implies stability of the uniform invariant density in the $BV$-sense.  

We first consider the setting of Theorem \ref{theo_main2}. In this case we studied the values $\varepsilon=0.1$ and $\varepsilon=0.2$. From Figure \ref{fig_psi0} we can read that there exists an invariant density $f_*(\varepsilon)$ for which $2+\varepsilon F\left(\frac{1}{E_{f_*(0.1)}}-2\right) \approx 2.0006$ and $2+\varepsilon F\left(\frac{1}{E_{f_*(0.2)}}-2\right) \approx 2.0181$. Running our simulations we can see from Table \ref{tab_x0} that in both cases $\overline{\text{var}}$ does not converge to zero, so the uniform density is not stable in $BV$-sense. To convince ourselves more thoroughly that the constant density is indeed not stable, we made computations with pools of initial densities very close to the constant one in the sense that $\text{var}(f^i_0) < 10^{-4}$. In this case convergence is slower, but it is clearly not to the constant density, see Table \ref{tab_x0b}.

\begin{table}
	\[
	\begin{array}{ | l | l | l | l | }
	\hline
	t & \overline{\text{var}} & \overline{\beta} \\ \hline
	0 & 7.1815 & 2.0065   \\ \hline
	5 &     1.1544 &     2.0039   \\ \hline
	10 &     0.9559 &     2.0020    \\ \hline
	15 &     0.9856 &     2.0015    \\ \hline
	20 &     0.9901 &     2.0012    \\ \hline
	25 &     0.9917 &     2.0011    \\ \hline
	30 &     0.9928 &     2.0010    \\ \hline
	35 &     0.9934 &     2.0009    \\ \hline
	40 &     0.9940 &     2.0008    \\ \hline
	45 &     0.9944 &     2.0008    \\ \hline
	50 &     0.9947 &     2.0007    \\ \hline
	55 &     0.9949 &     2.0007    \\ \hline
	60 &     0.9950 &     2.0007    \\ \hline
	65 &     0.9952 &     2.0007    \\ \hline
	70 &     0.9953 &     2.0007    \\ \hline
	75 &     0.9954 &     2.0007    \\ \hline
	80 &     0.9954 &     2.0007    \\ \hline
	85 &     0.9955 &     2.0007    \\ \hline
	90 &     0.9955 &     2.0006    \\ \hline
	95 &     0.9956 &     2.0006    \\ \hline
	100 &     0.9956 &     2.0006    \\ \hline
	\end{array}
	\qquad
	\begin{array}{ | l | l | l | l | }
	\hline
	t & \overline{\text{var}} & \overline{\beta} \\ \hline
	0 & 5.1121 & 1.9993   \\ \hline
	5 &     0.9983 &     2.0130   \\ \hline
	10 &     0.9415 &     2.0117   \\ \hline
	15 &     0.9490 &     2.0134   \\ \hline
	20 &     0.9429 &     2.0149   \\ \hline
	25 &     0.9371 &     2.0160   \\ \hline
	30 &     0.9341 &     2.0165   \\ \hline
	35 &     0.9323 &     2.0168   \\ \hline
	40 &     0.9330 &     2.0171   \\ \hline
	45 &     0.9320 &     2.0174   \\ \hline
	50 &     0.9309 &     2.0176   \\ \hline
	55 &     0.9298 &     2.0178   \\ \hline
	60 &     0.9289 &     2.0179   \\ \hline
	65 &     0.9280 &     2.0180   \\ \hline
	70 &     0.9274 &     2.0181   \\ \hline
	75 &     0.9270 &     2.0181   \\ \hline
	80 &     0.9269 &     2.0181   \\ \hline
	85 &     0.9269 &     2.0181   \\ \hline
	90 &     0.9269 &     2.0181   \\ \hline
	95 &     0.9270 &     2.0181   \\ \hline
	100 &     0.9270 &     2.0181   \\ \hline
	\end{array}
	\]
	\caption{Computation of the mean total variation $\overline{\text{var}}$ and mean slope for $F(x)=x$. $T=100$, $M=K_1=K_2=10$. Left hand side: $\varepsilon=0.1$, right hand side: $\varepsilon=0.2$.} \label{tab_x0}
\end{table}

\begin{table}
	\[
	\begin{array}{ | l | l | l | l | }
	\hline
	t & \overline{\text{var}} & \overline{\beta} \\ \hline
	0 & 0.0000 & 2.0000   \\ \hline
	25 &     1.0000 &     2.0000   \\ \hline
	50 &     1.0000 &     2.0000    \\ \hline
	75 &     1.0000 &     2.0000    \\ \hline
	100 &     0.9998 &     2.0000    \\ \hline
	125 &     0.9995 &     2.0001    \\ \hline
	150 &     0.9987 &     2.0002    \\ \hline
	175 &     0.9978 &     2.0003    \\ \hline
	200 &     0.9969 &     2.0005    \\ \hline
	225 &     0.9966 &     2.0005    \\ \hline
	250 &     0.9966 &     2.0005    \\ \hline
	275 &     0.9964 &     2.0005    \\ \hline
	300 &     0.9962 &     2.0006    \\ \hline
	\end{array}
	\qquad
	\begin{array}{ | l | l | l | l | }
	\hline
	t & \overline{\text{var}} & \overline{\beta} \\ \hline
	0 & 0.0000 & 2.0000   \\ \hline
	25 &     1.0000 &     2.0000   \\ \hline
	50 &     0.9993 &     2.0003    \\ \hline
	75 &     0.9783 &     2.0072    \\ \hline
	100 &     0.9362 &     2.0160    \\ \hline
	125 &     0.9307 &     2.0177    \\ \hline
	150 &     0.9268 &     2.0181    \\ \hline
	175 &     0.9269 &     2.0181    \\ \hline
	200 &     0.9269 &     2.0181    \\ \hline
	225 &     0.9269 &     2.0181   \\ \hline
	250 &     0.9269 &     2.0181    \\ \hline
	275 &     0.9269 &     2.0181    \\ \hline
	300 &     0.9269 &     2.0181    \\ \hline
	\end{array}
	\]
	\caption{Computation of the mean total variation $\overline{\text{var}}$ and mean slope for $F(x)=x$, $\text{var}(f^i_0) < 10^{-4}$. $T=300$, $M=K_1=K_2=10$. Left hand side: $\varepsilon=0.1$, right hand side: $\varepsilon=0.2$.} \label{tab_x0b}
\end{table}

\begin{table}
	\[
	\begin{array}{ | l | l | l | l | }
	\hline
	t & \overline{\text{int}}_{f_*(0.1)} & \overline{\text{int}}_{f_*(0.2)} \\ \hline
	0 & 0.4096 & 0.5018   \\ \hline
	10 &     0.0101 &  0.0321     \\ \hline
	20 &     0.0049 &  0.0174       \\ \hline
	30 &     0.0028 &  0.0097     \\ \hline
	40 &     0.0017 &  0.0064      \\ \hline
	50 &     0.0011 &  0.0040      \\ \hline
	60 &     0.0007 &  0.0021       \\ \hline
	70 &     0.0005 &  0.0007       \\ \hline
	80 &     0.0004 &  0.0003       \\ \hline
	90 &     0.0003 &  0.0002       \\ \hline
	100 &     0.0002 &  0.0001       \\ \hline
	110 &     0.0002 &  0.0001       \\ \hline
	120 &     0.0002 &  0.0001       \\ \hline
	130 &     0.0001 &  0.0001       \\ \hline
	140 &     0.0001 &  0.0001       \\ \hline
	150 &     0.0001 &  0.0001      \\ \hline
	160 &     0.0001 &  0.0001       \\ \hline
	170 &     0.0001 &  0.0001       \\ \hline
	180 &     0.0001 &  0.0001       \\ \hline
	190 &     0.0001 &  0.0001       \\ \hline
	200 &     0.0001 &  0.0001       \\ \hline
	\end{array}
	\]
	\caption{Computation of $\overline{\text{int}}_{f_*(\varepsilon)}$ and mean slope for $F(x)=x$, $T=200$, $\bar{T}=5000$. $M=K_1=K_2=10$. } \label{tab_x0c}
\end{table}

On the other hand, we can also read form Table \ref{tab_x0} that $\overline{\beta}$ converges to $2.0006$ and $2.0181$ for $\varepsilon=0.1$ and $0.2$ respectively, and this suggests that the nontrivial invariant densities are stable. Table \ref{tab_x0c} provides further evidence pointing to the stability of the nontrivial invariant densities in the $L^1$-sense. Note that convergence of $\overline{\text{int}}_{f_*(\varepsilon)}$ to zero is not something to be expected since $\tilde{f}_{\bar{T}}(\varepsilon)$ is just an approximation of $f_*(\varepsilon)$. However, as $\bar{T}$ is larger, $\overline{\text{int}}_{f_*(\varepsilon)}$ converges to smaller values which supports our hypothesis.

So it seems likely that the constant density looses its stability (in the $L^1$-sense) as $\varepsilon$ becomes larger than zero, and a new stable invariant density arises (in the $L^1$-sense).

\begin{figure}[h]
	\centering
	\begin{subfigure}[b]{0.44\textwidth}
		\centering
		\includegraphics[scale=0.5]{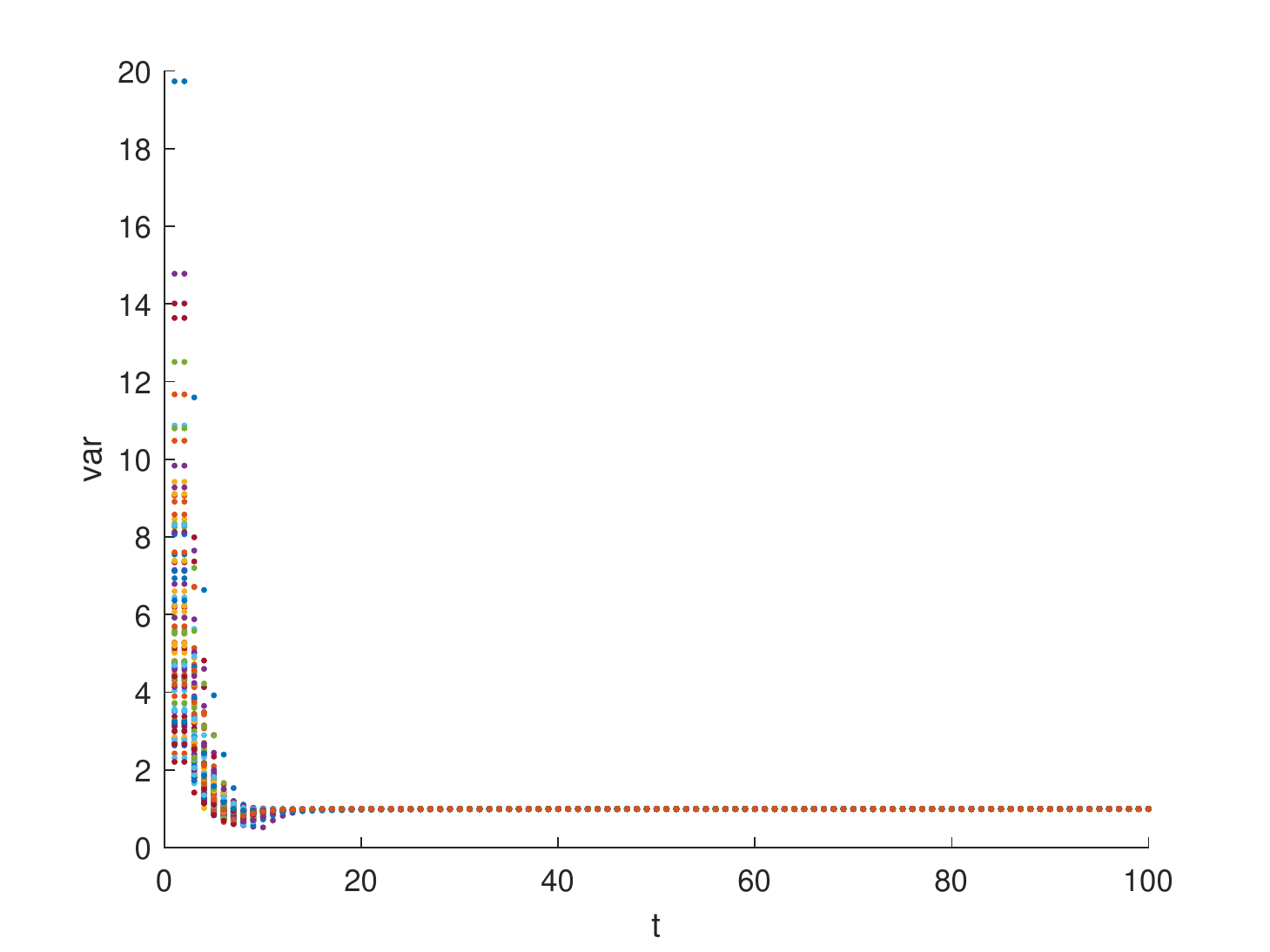}
		\caption{$\varepsilon=0.1$} 
	\end{subfigure}
	\qquad
	\begin{subfigure}[b]{0.44\textwidth}
		\centering
		\includegraphics[scale=0.5]{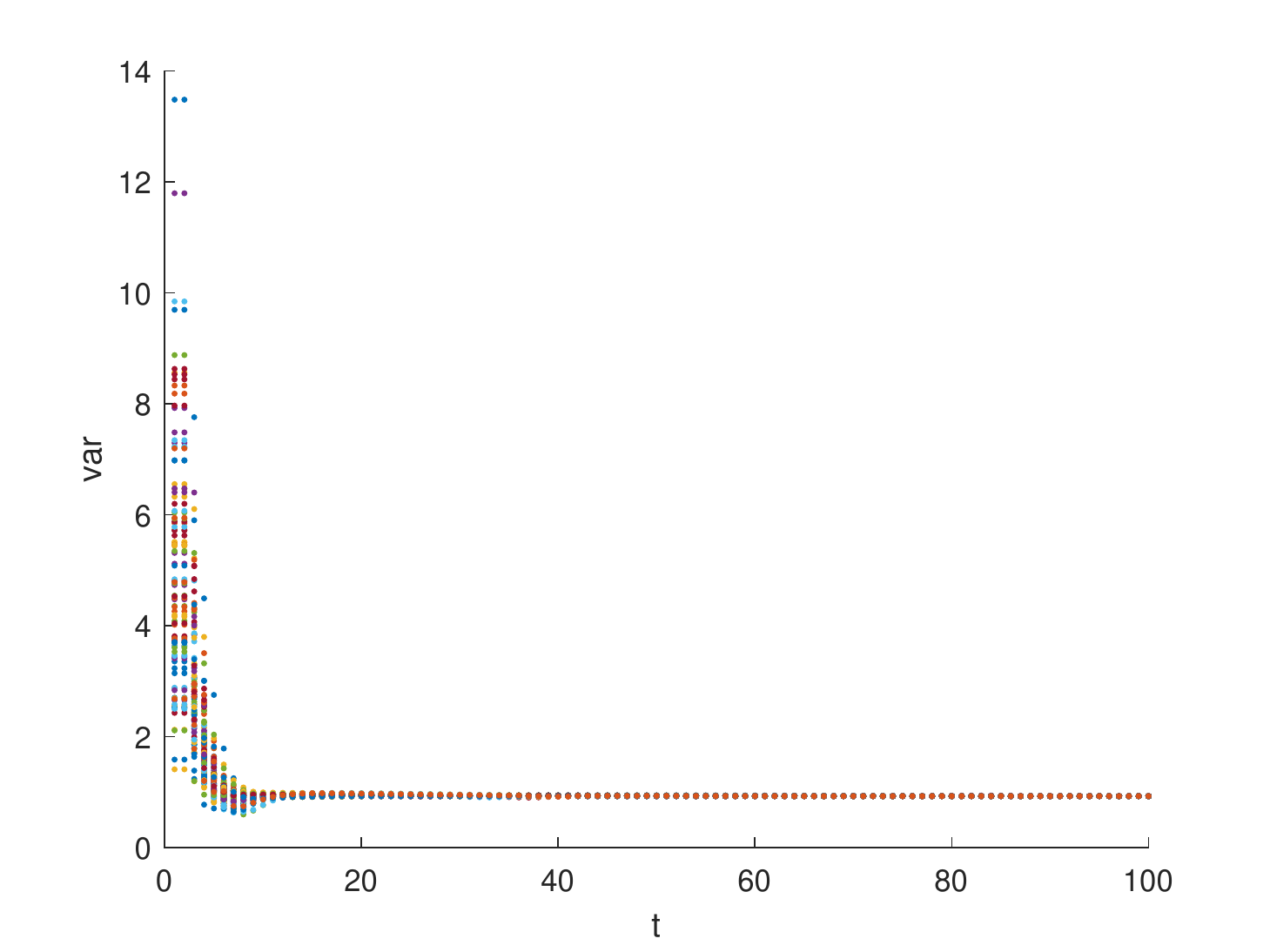}
		\caption{$\varepsilon=0.2$} 
	\end{subfigure}
	\caption{Total variation for each $f_t^i$ as a function of time, $F(x)=x$, $T=100$, $K_1\times K_2=100$.} \label{fig_idcloudvar}
\end{figure}

\begin{figure}[h]
	\centering
	\begin{subfigure}[b]{0.44\textwidth}
		\centering
		\includegraphics[scale=0.5]{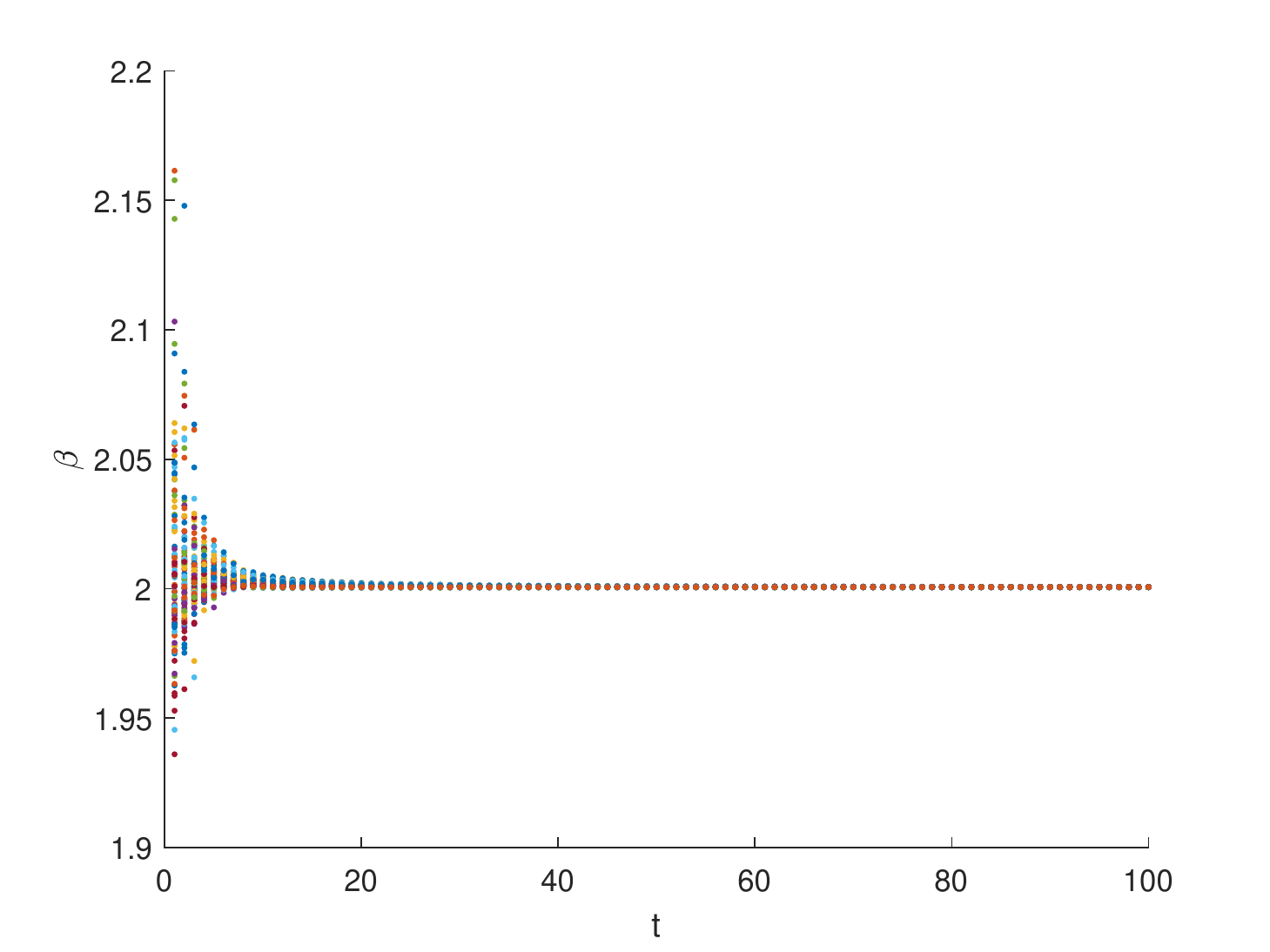}
		\caption{$\varepsilon=0.1$} 
	\end{subfigure}
	\qquad
	\begin{subfigure}[b]{0.44\textwidth}
		\centering
		\includegraphics[scale=0.5]{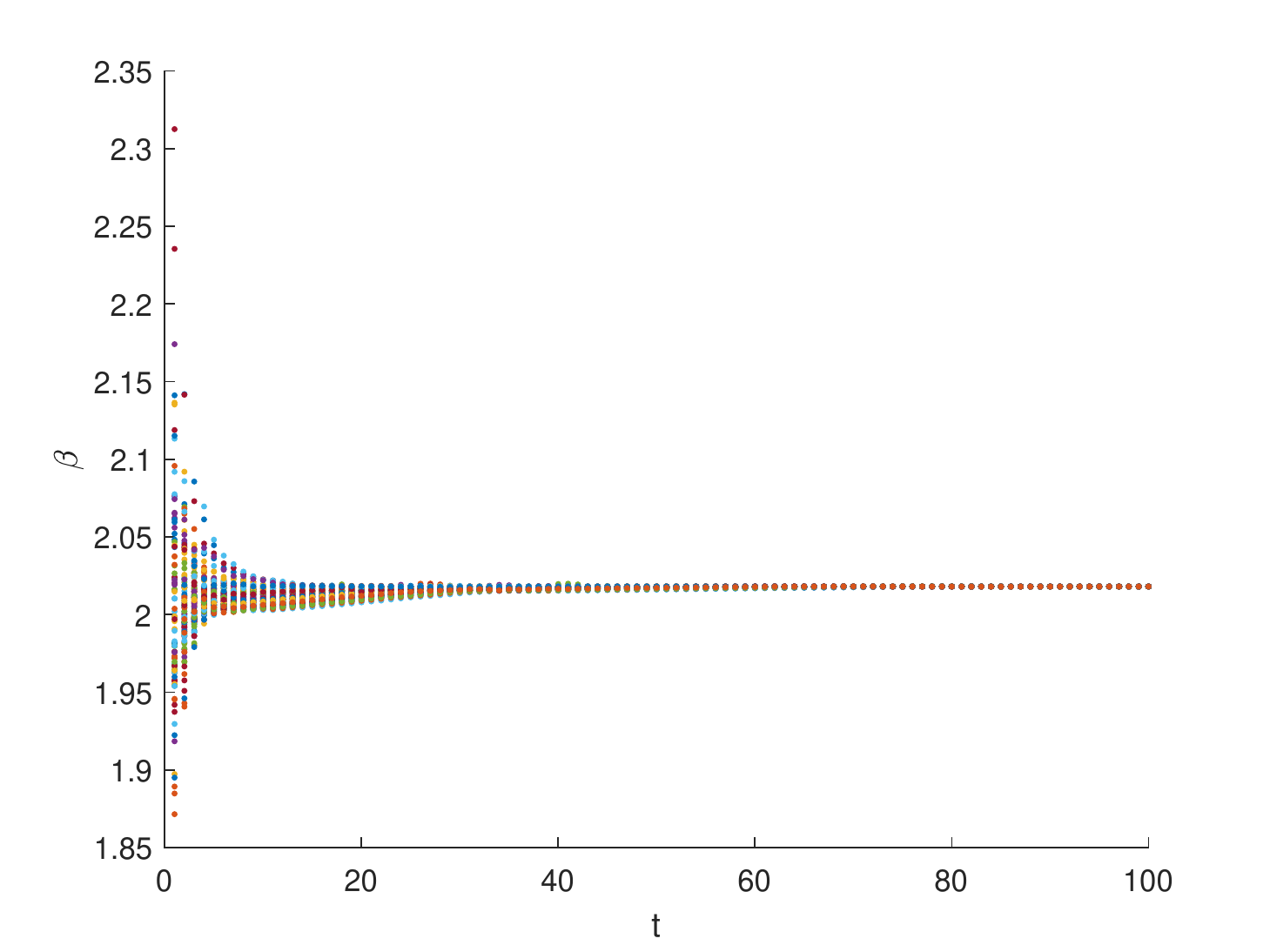}
		\caption{$\varepsilon=0.2$} 
	\end{subfigure}
	\caption{Associated slope to each $f_t^i$ as a function of time, $F(x)=x$, $T=100$, $K_1\times K_2=100$.} \label{fig_idcloudbeta}
\end{figure}

To back our assumption that $\overline{\text{var}}$ and $\bar{\beta}$ correctly describes the behavior of a typical density, we plot all the total variation and $\beta$ for the densities on Figures \ref{fig_idcloudvar} and \ref{fig_idcloudbeta} respectively, for each time instance. We can see that they all converge to a single value, hence the averaging does not give an average of different asymptotic behaviors but shows us the true one.

\begin{figure}[h]
	\centering
	\begin{subfigure}[b]{0.44\textwidth}
		\centering
		\includegraphics[scale=0.7]{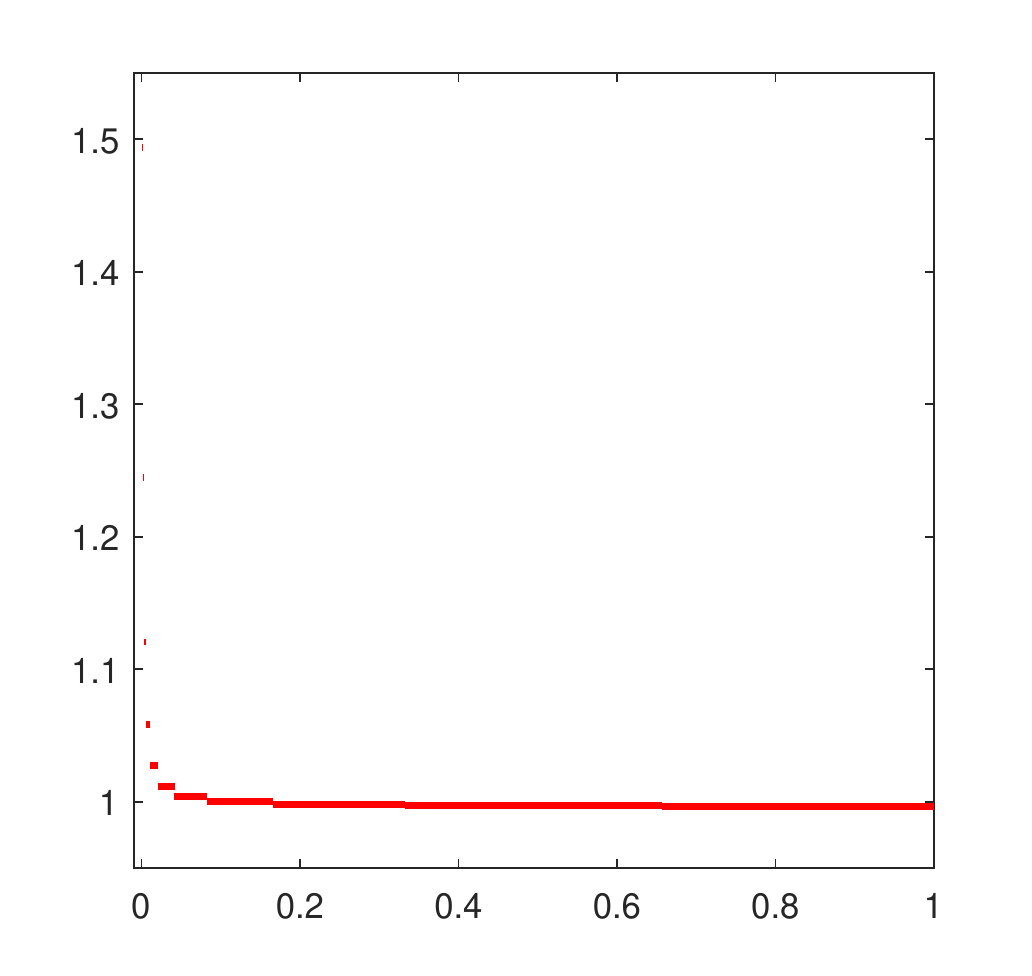}
		\caption{$\varepsilon=0.1$} 
	\end{subfigure}
	\qquad
	\begin{subfigure}[b]{0.44\textwidth}
		\centering
		\includegraphics[scale=0.7]{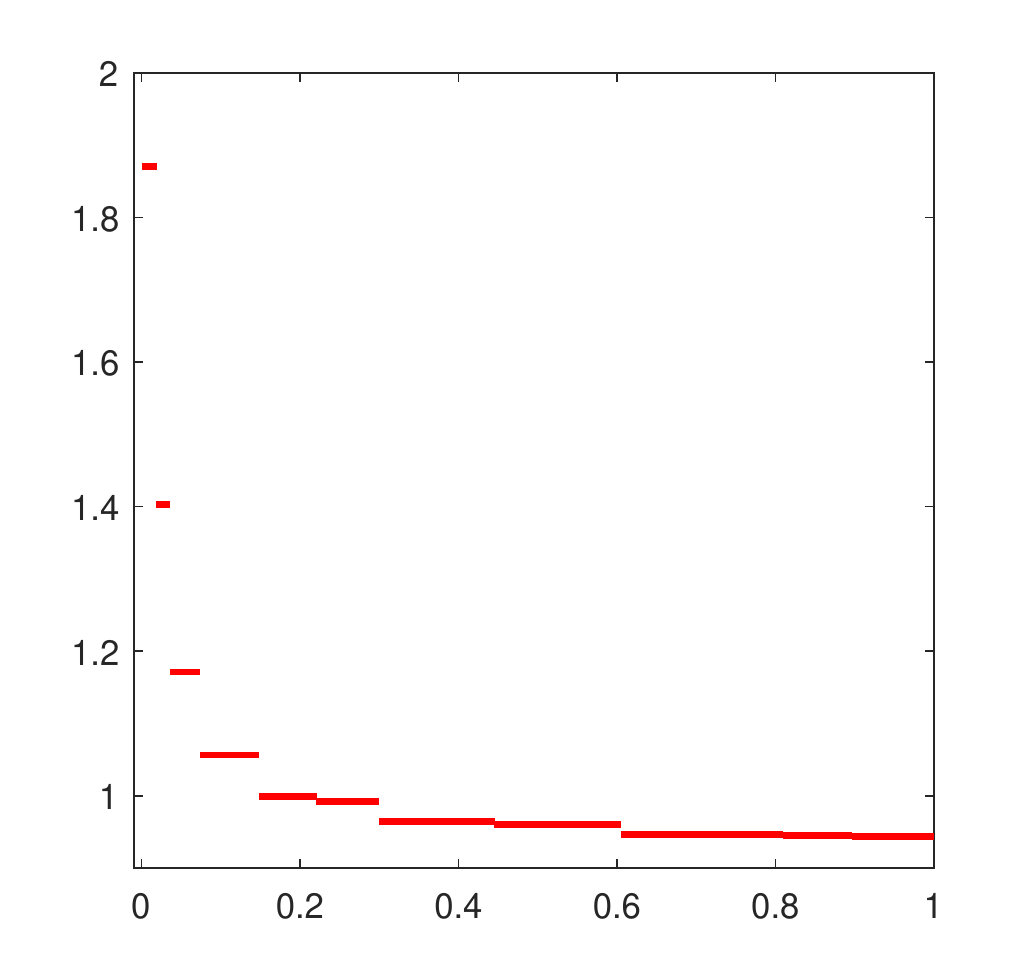}
		\caption{$\varepsilon=0.2$} 
	\end{subfigure}
	\caption{Approximation of the invariant density of the self-consistent system \eqref{eq_selfc} with $F(x)=x$ by a high iterate of an appropriate initial density.} \label{fig_densid}
\end{figure}

The asymptotic densities obtained from iterating an appropriate initial density for both cases $\varepsilon=0.1$ and $\varepsilon=0.2$ are pictured on Figure \ref{fig_densid}.

\begin{figure}[h]
	\centering
	\begin{subfigure}[b]{0.44\textwidth}
		\centering
		\includegraphics[scale=0.5]{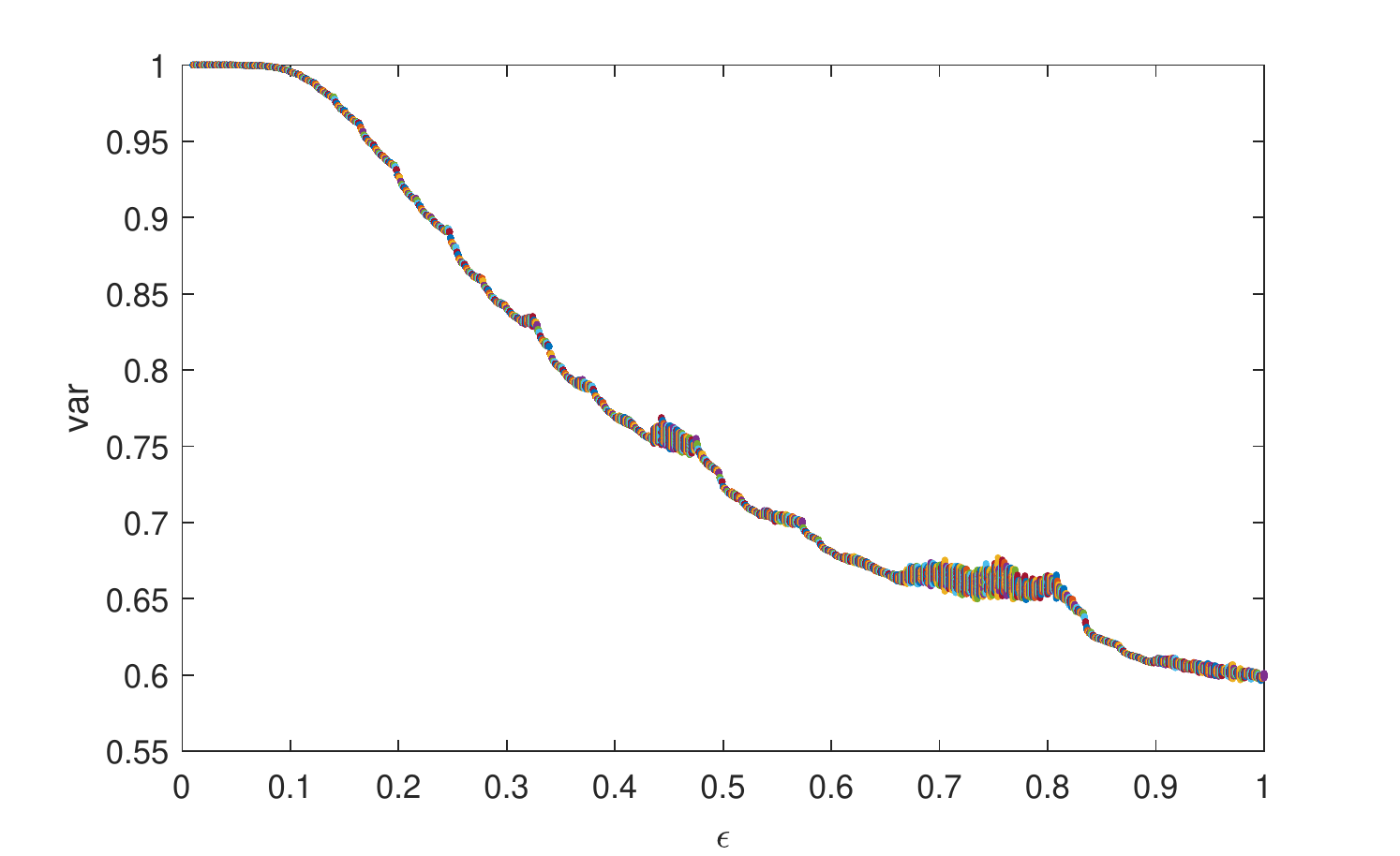}
	\end{subfigure}
	\qquad
	\begin{subfigure}[b]{0.44\textwidth}
		\centering
		\includegraphics[scale=0.5]{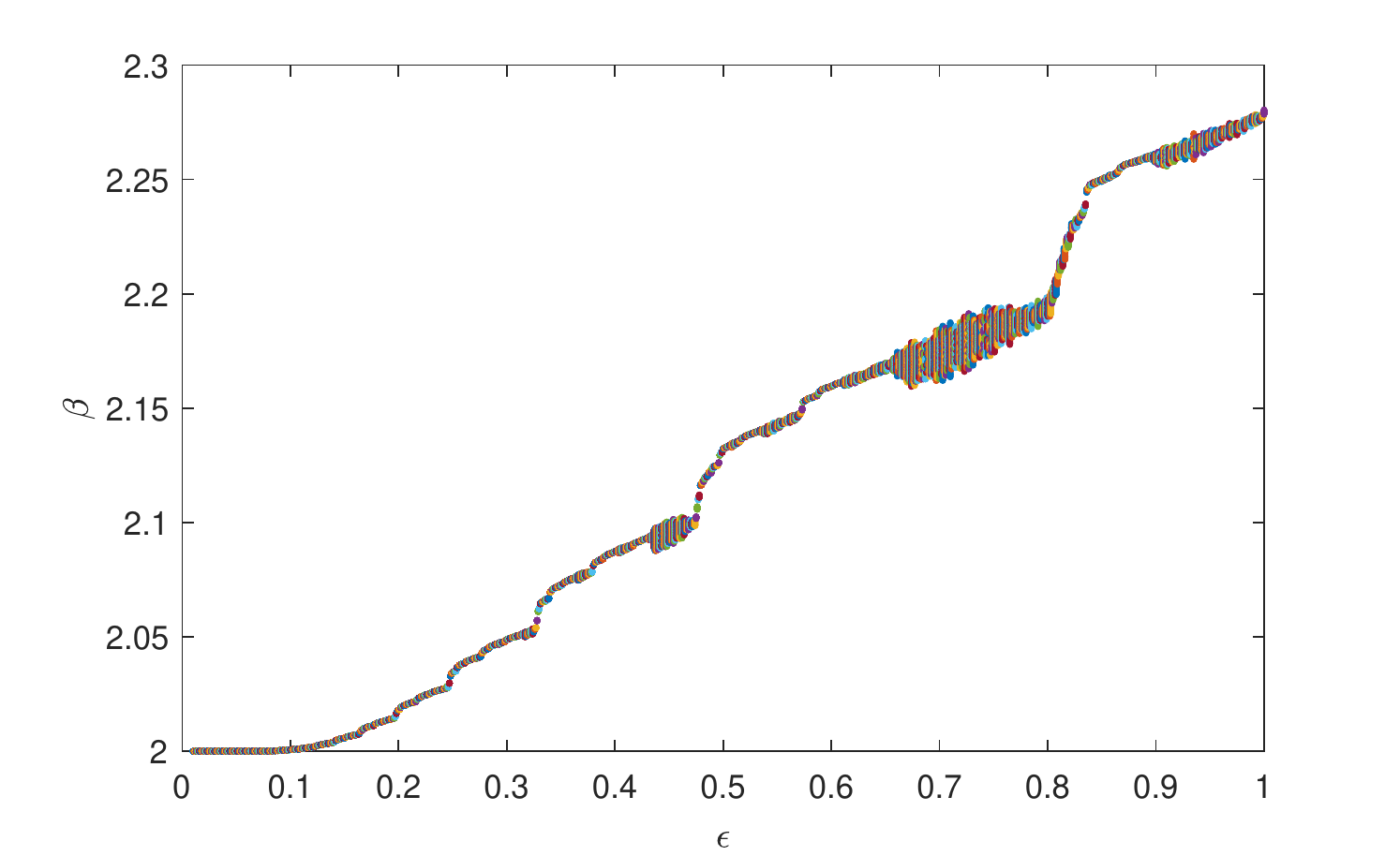}
	\end{subfigure}
	\caption{For each $\varepsilon= k \cdot 10^{-3}$, $k=1,\dots,10^3$ the values $\overline{\text{var}}(t)$ and $\overline{\beta}(t)$ are plotted for $t=150,\dots,200$, $F(x)=x$, $K_1 \times K_2=100$.} \label{fig_eps}
\end{figure}

Finally we discuss the \emph{convergence} of our method depending on the value of $\varepsilon$. By convergence we mean that the quantities $\overline{\text{var}}$ and $\overline{\beta}$ settle at a value $\overline{\text{var}}_*$ and $\overline{\beta}_*$ in the sense that 
\[
|\overline{\text{var}}_*-\overline{\text{var}}(t)| < 10^{-4} \quad \text{and} \quad |\overline{\beta}_*-\overline{\beta}(t)| < 10^{-4} \quad \text{for} \quad t=T_0,\dots,T_1,
\]
where we choose $T_0,T_1$ to be some fixed large numbers. In Figure \ref{fig_eps} we plotted the mean total variation and mean slope of the last 50 iterates of our density pool for a range of $\varepsilon$ values, that is, we chose $T_0=150$ and $T_1=200$. If this produces a considerable range of values for a single $\varepsilon$ (an interval having length larger then $10^{-4}$ above an $\varepsilon$ value), then the method does not converge. So we can read from this figure that our method is converges until $\varepsilon \approx 0.4$.

Now we move on to consider the setting of Theorem \ref{theo_main}. We studied the cases $F(x)=x^2$, $F(x)=x^4$ and $F(x)=x^6$ for a few values of $\varepsilon$ for which simulations similar to the ones discussed for the case of $F(x)=x$ clearly suggest unique or multiple absolutely continuous invariant measures. In the first columns of Tables \ref{tab_x2} we see computations in the cases when the constant density is the unique invariant one, see Figure \ref{fig_psi}. We can see that in all cases $\overline{\text{var}}(t)$ decreases rapidly, suggesting the stability of the invariant constant density in $BV$-sense. 

\begin{table}
	\[
	\begin{array}{ | l | l | l | l | }
	\hline
	t &  \varepsilon=1 & \varepsilon=2.5 \\ \hline 
	0 &  6.4694 & 5.0341 \\ \hline
	5 &          1.2483 & 1.1951   \\ \hline
	10 &          0.9808 & 0.8849   \\ \hline
	15 &          0.9914 & 0.8471  \\ \hline
	20 &          0.9976 & 0.8447  \\ \hline
	25 &          0.9994 & 0.8275  \\ \hline
	30 &          0.9999 & 0.8418  \\ \hline
	35 &          1.0000 & 0.8305  \\ \hline
	40 &          1.0000 & 0.8468  \\ \hline
	45 &          1.0000 & 0.8321  \\ \hline
	50 &          0.9531 & 0.8358 \\ \hline
	55 &          0.5357 & 0.8382  \\ \hline
	60 &          0.1802 & 0.7651  \\ \hline
	65 &          0.0128 & 0.6366  \\ \hline
	70 &          0.0004 & 0.5117  \\ \hline
	75 &          0.0000 & 0.4535  \\ \hline
	80 &          0.0000 & 0.4269  \\ \hline
	85 &          0.0000 & 0.4255  \\ \hline
	90 &          0.0000 & 0.4015  \\ \hline
	95 &          0.0000 & 0.3966  \\ \hline
	100 &      0.0000 & 0.3813  \\ \hline
	\end{array}
	\qquad
	\begin{array}{ | l | l | l | l | }
	\hline
	t & \varepsilon=1 &  \varepsilon=35 \\ \hline 
	0 & 5.3766 & 6.4806 \\ \hline
	5 &         1.2370 &            1.1785 \\ \hline
	10 &         0.9882 &            0.9788 \\ \hline
	15 &         0.6884 &              0.9873 \\ \hline
	20 &         0.0474 &              0.5994 \\ \hline
	25 &         0.0015 &           0.1093 \\ \hline
	30 &     0.0000 &              0.0347\\ \hline
	35 &     0.0000 &              0.0023 \\ \hline
	40 &     0.0000 &          0.0001 \\ \hline
	45 &     0.0000 &          0.0000 \\ \hline
	50 &     0.0000 &         0.0000 \\ \hline
	55 &     0.0000 &         0.0000 \\ \hline
	60 &     0.0000 &          0.0000 \\ \hline
	65 &     0.0000 &          0.0000 \\ \hline
	70 &     0.0000 &          0.0000 \\ \hline
	75 & 0 &         0.0000 \\ \hline
	80 & 0 &  0.0000 \\ \hline
	85 & 0 &  0 \\ \hline
	90 & 0 &  0 \\ \hline
	95 & 0 &  0 \\ \hline
	100 & 0 &  0 \\ \hline
	\end{array}
	\qquad
	\begin{array}{ | l | l | l | l | }
	\hline
	t &  \varepsilon=1 & \varepsilon=400 \\ \hline 
	0 &   7.7582 & 5.1444 \\ \hline
	5 &         1.3852 &     1.0419 \\ \hline
	10 &         0.5224 &     0.5990 \\ \hline
	15 &          0.0210 &     0.1435 \\ \hline
	20 &          0.0007 &     0.0026 \\ \hline
	25 &          0.0000 &     0.0002 \\ \hline
	30 &          0.0000 &     0.0000 \\ \hline
	35 &          0.0000 &     0.0000 \\ \hline
	40 &          0.0000 &     0.0000 \\ \hline
	45 &          0.0000 &     0.0000 \\ \hline
	50 &          0.0000 &     0.0000 \\ \hline
	55 &         0.0000 &     0.0000 \\ \hline
	60 &          0.0000 &     0.0000 \\ \hline
	65 &          0 &     0.0000 \\ \hline
	70 &          0 &     0 \\ \hline
	75 &        0 &     0 \\ \hline
	80 &  0 &     0\\ \hline
	85 &  0 &     0 \\ \hline
	90 &  0 &     0 \\ \hline
	95 &  0 & 0 \\ \hline
	100 &  0 & 0 \\ \hline
	\end{array}
	\]
	\caption{Computation of the mean total variation $\overline{\text{var}}$. $T=100$, $M=K_1=K_2=10$. Left hand side: $F(x)=x^2$, center: $F(x)=x^4$, right hand side: $F(x)=x^6$.} \label{tab_x2}
\end{table}

In the second columns of Tables \ref{tab_x2}, we considered situations where multiple invariant densities exist, see Figure \ref{fig_psi20} (a) and Figure \ref{fig_psi30}. In case of $F(x)=x^2$, our method does not converge in the sense that the values of $\overline{\text{var}}$ do not settle at a value with precision $10^{-4}$ if we consider any subinterval $t=T_0,\dots,100$. We did other experiments with $T_1=500$ and $T_1=1000$, $T_0=T_1-100$ but obtained similar results. However, a slow decrease of $\overline{\text{var}}$ is observable, so we do not exclude the possibility that we could obtain convergence for $T_1=10^k$ for some $k$ large, but our limited resources prohibit us from carrying out such computations in a reasonable amount of time. 

On the other hand, for $F(x)=x^4$ and $F(x)=x^6$ we see fast convergence of $\overline{\text{var}}$ to zero implying that the constant density is likely to be stable one in $BV$-sense.

\begin{figure}[h]
	\centering
	\begin{subfigure}[b]{0.44\textwidth}
		\centering
		\includegraphics[scale=0.4]{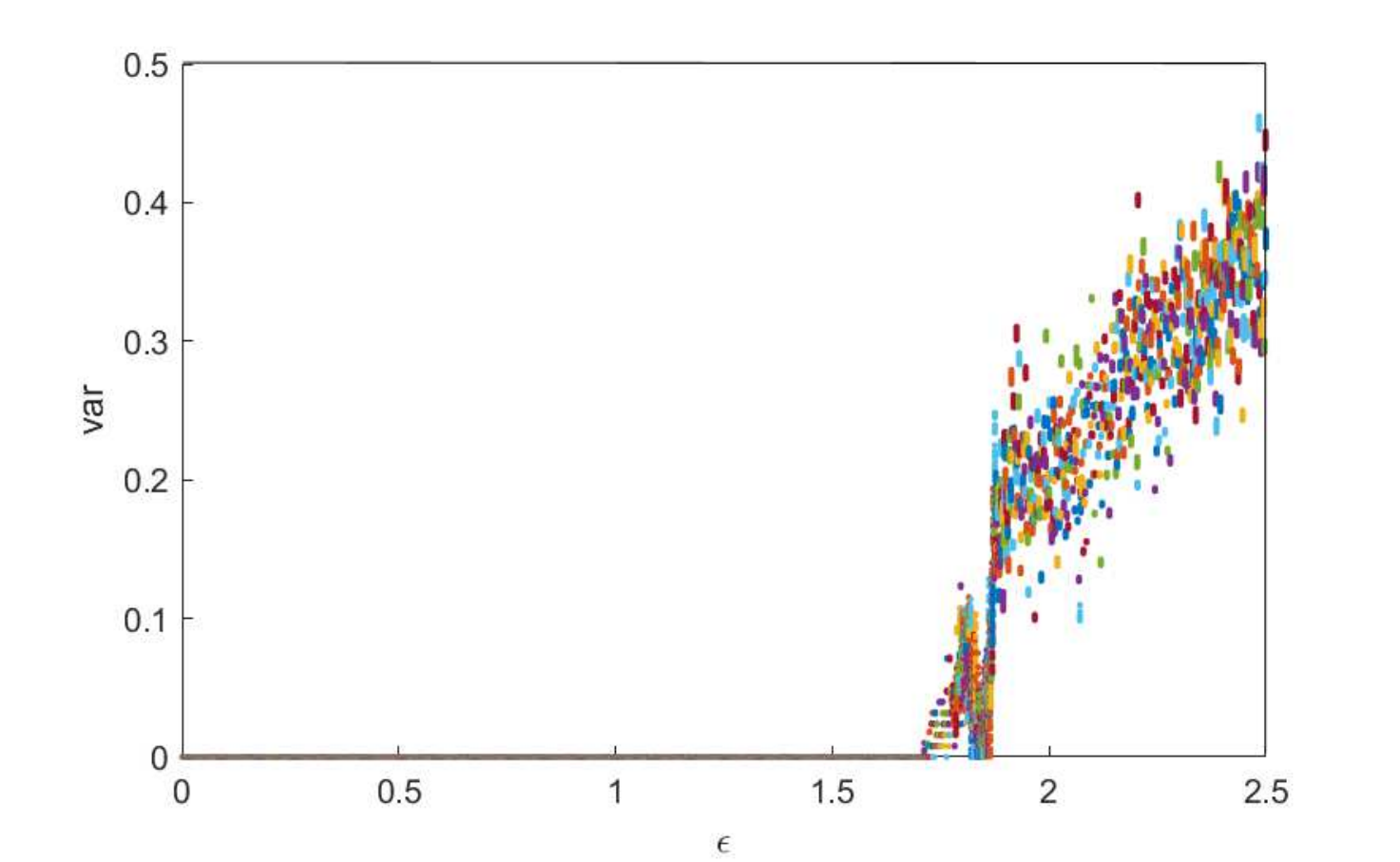}
		\caption{$F(x)=x^2$, $\Delta=10^{-3}$, $E=2.5$} 
	\end{subfigure}
	\qquad
	\begin{subfigure}[b]{0.44\textwidth}
		\centering
		\includegraphics[scale=0.4]{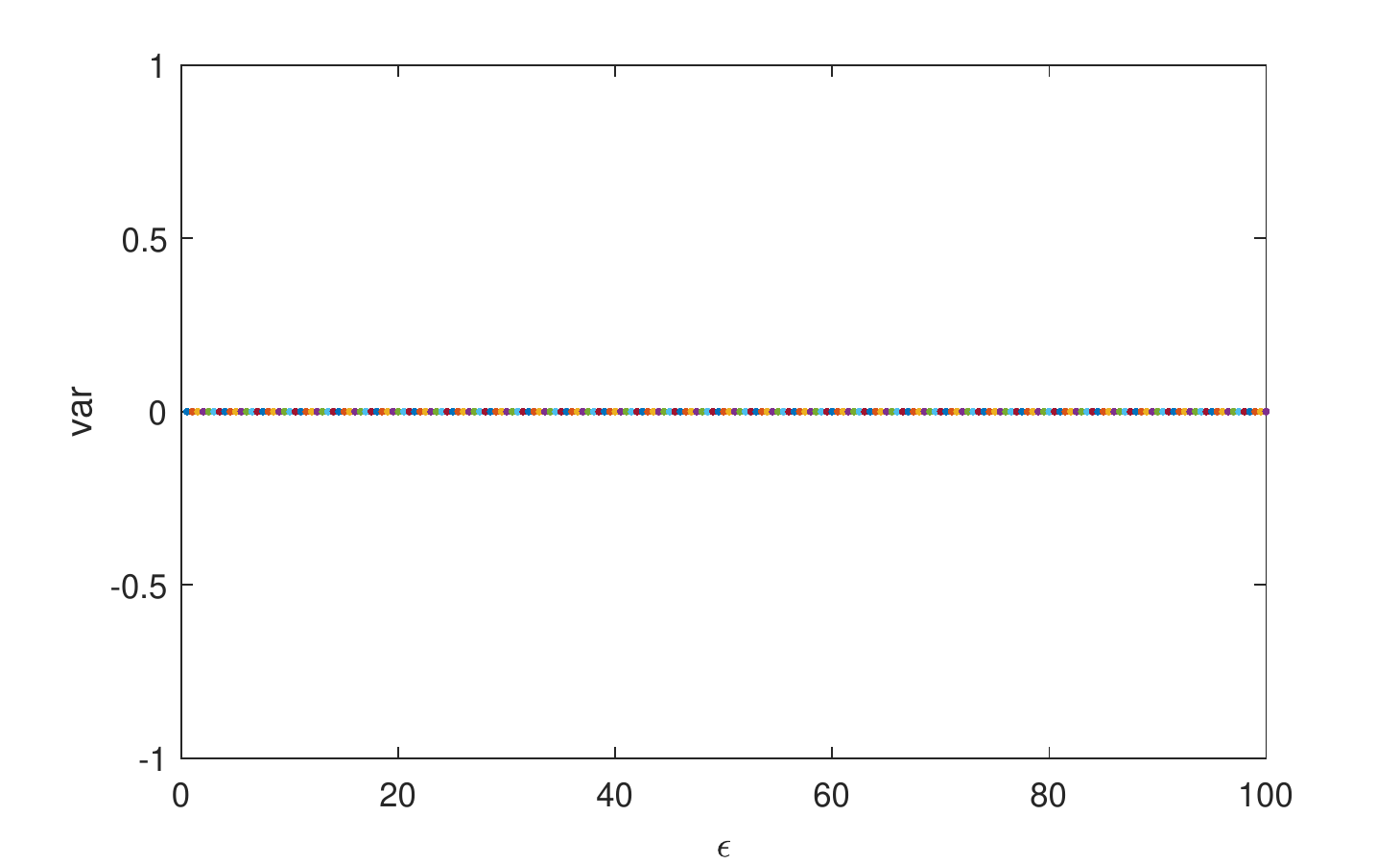}
		\caption{$F(x)=x^4$, $\Delta=0.5$, $E=100$} 
	\end{subfigure}
	\qquad
	\begin{subfigure}[b]{0.44\textwidth}
		\centering
		\includegraphics[scale=0.4]{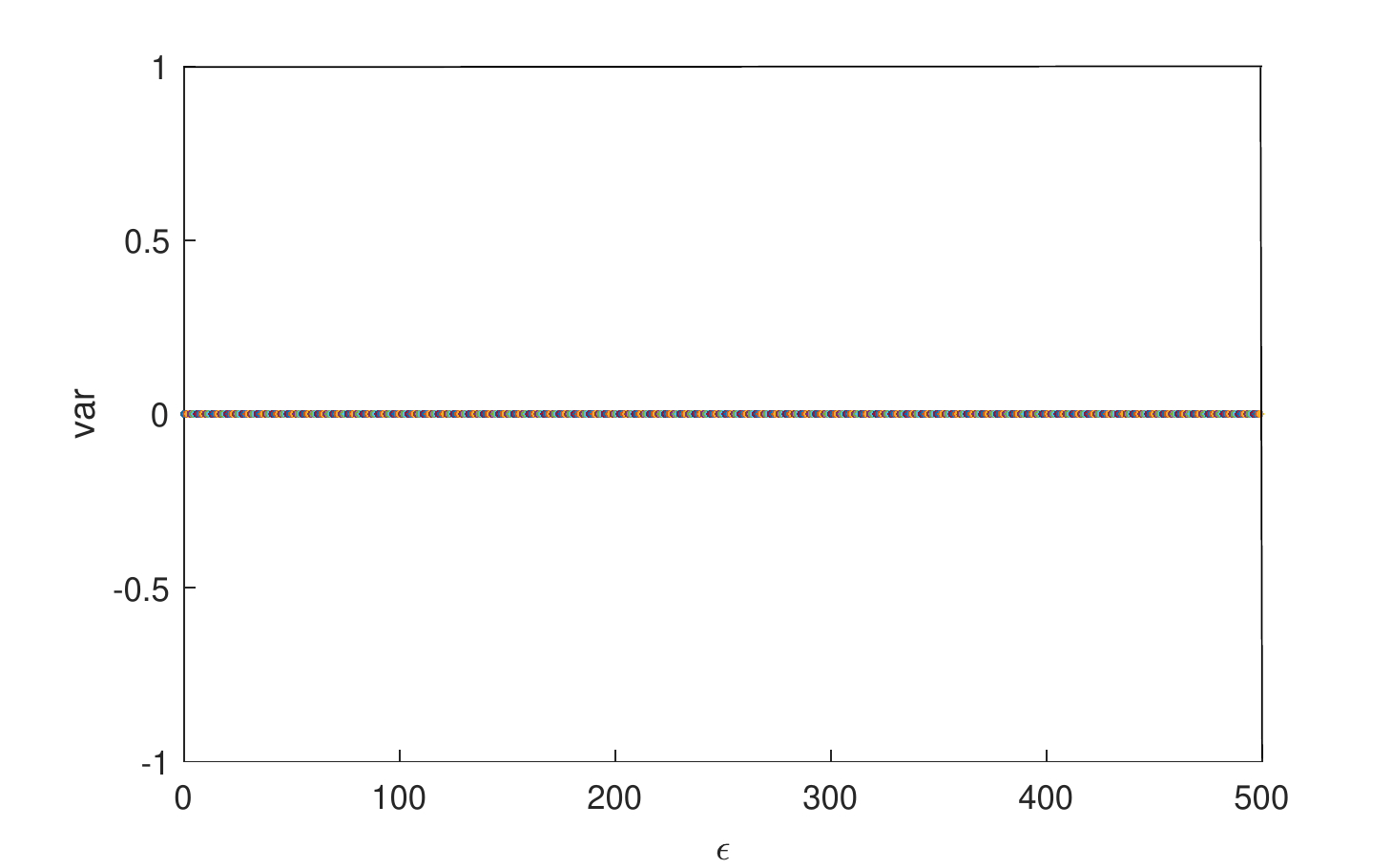}
		\caption{$F(x)=x^6$, $\Delta=10^{-2}$, $E=12$} 
	\end{subfigure}
	\caption{For each $\varepsilon= k \cdot \Delta$, $k=1,\dots,E/\Delta$ the values $\overline{\text{var}}(t)$ are plotted for $t=150,\dots,200$, $K_1 \times K_2=100$.} \label{fig_eps2}
\end{figure}

On Figure \ref{fig_eps2} we discuss the convergence of our method for a range of $\varepsilon$ similarly as we have done for $F(x)=x$ (plotted on Figure \ref{fig_eps}). We can see that for $F(x)=x^2$ our method becomes erratic for larger $\varepsilon$ values than $\varepsilon \approx 1.6$. But our method seems to converge nicely in the $F(x)=x^4$ and $F(x)=x^6$ case for all values of $\varepsilon$ considered.

\FloatBarrier 

\section{Concluding remarks}

To answer the question regarding stability rigorously, the careful study of the self-consistent transfer operator is necessary. When $\varepsilon=0$, $$\mathcal{F}_{\varepsilon}^n=P_T^n,$$ where $P_T$ is the transfer operator of the doubling map. The stability result regarding the doubling map can be proved by elementary means: one can show by explicit calculations that $\|P_T^nf-\mathbf{1}\|_{BV}\leq \frac{C}{2^n}\|f-\mathbf{1}\|_{BV}$ for all $f \in BV([0,1],\mathbb{R})$ such that $\int_0^1 f=1$ However, when $\varepsilon > 0$ we have to deal with the self-consistency. In this case
\[
\mathcal{F}_{\varepsilon}^n=P^{\varepsilon}_{f_0}P^{\varepsilon}_{f_1}\dots P^{\varepsilon}_{f_{n-1}}
\]
for some densities $f_0,\dots f_{n-1}$, so the problem does not simplify to the study of a single linear operator. Provided that $\varepsilon$ is small enough, it is natural to expect that $P^{\varepsilon}_{f_0}P^{\varepsilon}_{f_1}\dots P^{\varepsilon}_{f_{n-1}}$ is `close' to $P_T^n$ in some sense, hence acts similarly. In the coupled map systems of \cite{keller2000ergodic} and \cite{balint2018synchronization} giving rise to self-consistent dynamics this is precisely the strategy to prove stability in $BV$. However, the ($C^3$ or $C^2$ and Lipschitz second derivative) smoothness of the stepwise dynamics is an essential part of their proof. In the setting of this paper, the stepwise dynamics $T_{\mu}^{\varepsilon}$ is discontinuous, posing a major technical difficulty, so it can also be the case that different tools are needed to study the asymptotic behavior of the operator $\mathcal{F}_{\varepsilon}$. In correspondence to the numerical stability of the nontrivial invariant densities of the case $F(x)=x$, we also believe that in full generality only stability in the $L^1$-sense is to be expected.   

Another question that arises observing the Figures \ref{fig_psi0} and \ref{fig_psi20} is if the intersection of the numerical approximation of $\beta \mapsto \psi^{\varepsilon}(\beta)$ and the line $x=y$ approximates a single intersection of $\beta \mapsto \psi^{\varepsilon}(\beta)$ and $x=y$ or infinitely many accumulating ones. As the regularity of $\psi^{\varepsilon}$ is quite low, one can quite possibly imagine infinitely many intersections reminiscent of the infinitely many accumulating zeros of the trajectory of Brownian motion. This would be interesting, as it would give infinitely many absolutely continuous invariant measures.

\newpage

\bibliographystyle{alpha}
\bibliography{references_selfc}
\nocite{*}

\end{document}